\documentclass[11]{article}
\usepackage{fullpage}
\usepackage{hyperref}
\usepackage{float}
\usepackage{tikz}
\usetikzlibrary{arrows}

\usepackage{times}
\usepackage{epsfig}
\usepackage{epstopdf}
\usepackage{graphicx}
\usepackage{enumerate}
\usepackage{amsmath}
\usepackage{amssymb}
\usepackage[ruled,vlined]{algorithm2e}
\usepackage{bbm}
\usepackage{listings}
\usepackage{subcaption}
\usepackage{kantlipsum} 
\usepackage{mwe} 
\usepackage{amsthm}
\usepackage{physics}
\usepackage{enumitem} 
\usepackage{multirow}
\usepackage{algorithm2e}
\usepackage{ctable}

\newtheorem{theorem}{Theorem}
\newtheorem{lemma}{Lemma}
\newtheorem{remark}{Remark}
\newtheorem{definition}{Definition}
\newtheorem{assumption}{Assumption}[]

\newcommand{\qedwhite}{\hfill \ensuremath{\Box}}

 \usepackage{mathtools}

\usepackage[colorinlistoftodos,bordercolor=orange,backgroundcolor=orange!20,linecolor=orange,textsize=scriptsize]{todonotes}

\newcommand\ddfrac[2]{\frac{\displaystyle #1}{\displaystyle #2}}
\makeatletter

\makeatletter
\renewcommand\paragraph{\@startsection{paragraph}{4}{\z@}%
            {-2.5ex\@plus -1ex \@minus -.25ex}%
            {1.25ex \@plus .25ex}%
            {\normalfont\normalsize\bfseries}}
\makeatother
\author{Atal Narayan Sahu
\thanks{Atal Narayan Sahu is with the Extreme Computing Research Center, Division of Computer, Electrical and Mathematical Sciences and Engineering (CEMSE) at King Abdullah University of Science and Technology, Thuwal, Saudi Arabia-23955-6900, 
e-mail: atal.sahu@kaust.edu.sa.}
\quad Aritra Dutta \thanks{Aritra Dutta is with the  Extreme Computing Research Center, Division of Computer, Electrical and Mathematical Sciences and Engineering (CEMSE) at King Abdullah University of Science and Technology, Thuwal, Saudi Arabia-23955-6900, 
e-mail: aritra.dutta@kaust.edu.sa. Corresponding author.}
\quad Aashutosh Tiwari 
\thanks{Aashutosh Tiwari is with the Department of Mathematics  at Indian Institute of Technology, Kanpur, email: ashut669@gmail.com}
    \quad  Peter Richt\'{a}rik
        \thanks{Peter Richt\'{a}rik is with the Visual Computing Center, Division of Computer, Electrical and Mathematical Sciences and Engineering (CEMSE) at King Abdullah University of Science and Technology, Saudi Arabia-23955-6900, and MIPT, Russia;
e-mail: peter.richtarik@kaust.edu.sa}
}

\newcommand{\eqdef}{:=}
\newcommand{\R}{\mathbb{R}}

\newcommand{\Exp}{\mathbb{E}}
\newcommand{\E}[1]{{\mathbb{E}\left[#1\right] }}    

\newcommand{\cD}{{\cal D}}

\newcommand{\cL}{{\cal L}}

\newcommand{\cO}{{\cal O}}

\newcommand{\cX}{{\cal X}}

\newcommand{\mA}{{\bf A}}
\newcommand{\mB}{{\bf B}}

\newcommand{\mH}{{\bf H}}
\newcommand{\mI}{{\bf I}}

\newcommand{\mM}{{\bf M}}

\newcommand{\mS}{{\bf S}}

\newcommand{\mU}{{\bf U}}

\newcommand{\mW}{{\bf W}}

\newcommand{\mZ}{{\bf Z}}

\begin{document}
\date{}
\title{On the Convergence Analysis of Asynchronous SGD for Solving Consistent Linear Systems}
\maketitle

\begin{abstract}
 In the realm of big data and machine learning, data-parallel, distributed stochastic algorithms have drawn significant attention in the present days.~While the synchronous versions of these algorithms are well understood in terms of their convergence, the convergence analyses of their asynchronous counterparts are not widely studied. In this paper, we propose and analyze a {\it distributed, asynchronous parallel} SGD in light of solving an arbitrary consistent linear system by reformulating the system into a stochastic optimization problem as studied by Richt\'{a}rik and Tak\'{a}\~{c} in \cite{richtarik2017stochastic}. We compare the convergence rates of our asynchronous SGD algorithm with the synchronous parallel algorithm proposed by Richt\'{a}rik and Tak\'{a}\v{c} in~\cite{richtarik2017stochastic} under different choices of the hyperparameters---the stepsize, the damping factor, the number of processors, and the delay factor. We show that our asynchronous parallel SGD algorithm also enjoys a global linear convergence rate, similar to the {\em basic} method and the synchronous parallel method in \cite{richtarik2017stochastic} for solving any arbitrary consistent linear system via stochastic reformulation. We also show that our asynchronous parallel SGD improves upon the {\em basic} method with a better convergence rate when the number of processors is larger than four. We further show that this asynchronous approach performs asymptotically better than its synchronous counterpart for certain linear systems. Moreover, for certain linear systems, we compute the minimum number of processors required for which our asynchronous parallel SGD is better, and find that this number can be as low as two for some ill-conditioned problems.
\end{abstract}



\tableofcontents

\section{Introduction}\label{intro}
In the era of big data and artificial intelligence, optimization problems have become increasingly complex in nature. Although the computers are now more powerful and inexpensive, the problems have grown continuously larger in size and they are difficult to be maneuvered by a single processor. Owing to the nature of the high-volume of the data, an emerging interest is to device and analyze scalable, parallel, and distributed algorithms that can handle the data more efficiently as compared to the traditional optimization algorithms designed to run on a single processor.~To deal with the large-scale data, these new class of algorithms can take advantage of a multi-processor system where each processor has access to its own data-partition and it processes the data-partitions in {\em mini-batches}. 
Shalev-Shwartz et al. \cite{Shalev-Shwartz_pegasos} in 2007 and Gimpel et al. \cite{Gimpel} in 2010, explored the idea of mini-batches for stochastic algorithms in both the serial and parallel settings. In 2011, Dekel et al. \cite{Dekel} proposed a distributed mini-batch algorithm~(for online predictions)---a method that converts many serial gradient-based online prediction algorithms into distributed algorithms with an asymptotically optimal regret bound. However, in a distributed environment, the synchronous parallel algorithms tend to slow down due to unpredictable communication faults, significant network latency, and processors with different processing speeds. 

To overcome the above issues posed by the synchronous parallel algorithms in a distributed environment, there has been a recent focus on developing and analyzing \textit{asynchronous} algorithms. In asynchronous algorithms, processors with different storage capacity and processing speeds perform updates without synchronizing with others. 
Asynchronous algorithms were first introduced by Chazan and Miranker on chaotic relaxation in 1969 \cite{chazan} (also, see Frommer and Szyld \cite{frommer} and \cite{bertsekas2003parallel}). However, not only the inherent dynamics of asynchronous algorithms are challenging compared to their synchronous counterparts, but also their convergence analyses are much more mathematically involved. Historically, in the literature, the comparisons between the convergence rates of the asynchronous algorithms and their synchronous counterparts are also not vastly present and less understood. In this paper, our goal is to understand the convergence rates of the asynchronous and synchronous SGD in a fairly simple set-up. However, before discussing problem formulation, set-up, and contribution, we start with a brief overview of stochastic optimization. 

\paragraph*{Stochastic optimization} In machine learning and data-fitting applications, stochastic optimization is a broadly studied field. Consider the {\it stochastic optimization problem}:
\begin{eqnarray}\label{stochasticoptimization_problem}
 \min_{x\in\R^n}f(x)=\min_{x\in\R^n}\Exp_{\mS\sim\cD}[f_{\mS}(x)],
\end{eqnarray}
where $\cD$ is an user inferred distribution and $\mS$ is a random sample drawn from that distribution.~In supervised machine learning or deep learning, the above problem is known as {\em empirical risk minimization}~(ERM) problem: 
\begin{eqnarray}\label{distributed_stochasticoptimization_problem}
 \min_{x\in\R^n}[f(x)=\sum_{i=1}^n\frac{1}{n}\underbrace{\Exp_{\mS_i\sim\cD_i}[f_{\mS_i}(x)}_{\eqdef f_i(x)}]],
\end{eqnarray} 
where $f_i(x)$'s are instantiated by different distributions $\cD_i$ and $\mS_i$ is sampled from $\cD_i$.
In the distributed setting, $n$ in the ERM problem denotes number of processors/workers. Therefore, \eqref {distributed_stochasticoptimization_problem} is an important problem from the deep learning perspective as it captures data-parallelism~(distributed over $n$ processors (GPUs/CPUs etc.)). One of the most popular algorithms for solving \eqref{stochasticoptimization_problem} is the stochastic gradient descent (SGD)~\cite{robins_monro}. 
 For a given sequence of stepsize parameters $\{\omega_k\}$ with $\omega_k>0$ and the sequence of iterates $\{x_k\}$, the updates of SGD take the form:
\begin{eqnarray}\label{sgd}
 x_{k+1}=x_k-\omega_k \nabla f_{\mS_k}(x_k),
\end{eqnarray}
where $\nabla f_{\mS_k}(x_k)$ is the stochastic gradient arising from the sample $\mS_k\sim \cD$ drawn afresh in each iteration and is an unbiased estimator of the gradient of $f$. A natural next direction in solving \eqref{stochasticoptimization_problem} is to design a synchronized parallel update by using SGD \cite{parallel_SGD}. 
If more than one processors are available, then they can work simultaneously and each one of them can calculate a stochastic gradient independent of the other processors. At the end, the user can average the stochastic gradients from all processors to obtain the update as:
\begin{eqnarray}\label{parallel_sgd}
 x_{k+1}=x_k-\frac{\omega_k}{\tau}\sum_{i=1}^{\tau} \nabla f_{\mS_{ki}}(x_k),
\end{eqnarray}
where $\nabla f_{\mS_{ki}}(x_k)$ is the stochastic gradient arising from the independent sample $\mS_{ki}\sim \cD$ from each of the $\tau$ processors and $\omega_k>0$ is a uniform stepsize across $k^{\rm th}$ iteration. Note that, if $\tau=1$, then the update scheme in \eqref{parallel_sgd} is \eqref{sgd}. 
In this paper, we compare the convergence rates of the asynchronous and synchronous algorithms in a fairly simple set-up of solving any arbitrary consistent linear systems by reformulating them into stochastic optimization problems. Before explaining our main contributions, we introduce the {\em stochastic reformulation of a linear system}. 

\subsection{Stochastic reformulation of a linear system}\label{sec:stochastic_reformulation}
In \cite{richtarik2017stochastic}, Richt\'{a}rik and Tak\'{a}\v{c} reformulated any arbitrary consistent linear system into a stochastic optimization problem~(see details in \cite{richtarik2017stochastic, gower2015stochastic}; also see \cite{gower2015randomized} ). Consider a linear system:
\begin{eqnarray}\label{linear_system}
  \mA x=b,
\end{eqnarray}
where $\mA\in\R^{m\times n}$ and $\mA\neq 0$. Let the set of solutions $\cL\eqdef \{x : \mA x=b\}$ be non-empty. Simply put, we consider a consistent linear system that has a solution but the solution is not necessarily unique. Richt\'{a}rik and Tak\'{a}\v{c}, reformulated \eqref{linear_system} into different stochastic problems and showed that for any arbitrary consistent linear system, the stochastic reformulations of \eqref{linear_system} are exact. In other words, the set of the solutions of any of those equivalent stochastic formulations is exactly the same as $\cL$ --- which they formally defined as the {\it exactness} assumption (see Assumption \ref{assump1}) in Section \ref{sec:2}. 
To motivate further, we will now introduce some technicalities. For a symmetric positive definite matrix $\mB$, denote $\langle \cdot\;,\;\cdot \rangle_{\mB}$ as the $\mB$-inner product and let $\|x\|_{\mB}=\sqrt{x^\top\mB x}$ be the (semi)-norm induced by it. Therefore, by using the idea proposed in \cite{richtarik2017stochastic}, one can define a stochastic function $f_{\mS}(x)\eqdef\frac{1}{2}\|\mA x-b\|_{\mH}^2,$ where $\mH=\mS(\mS{^\top}\mA\mB^{-1}\mA^{\top}\mS)^{\dagger}\mS^{\top}$ is a random, symmetric, and positive definite matrix. Indeed minimizing \eqref{stochasticoptimization_problem} with $f_{\mS}(x)=\frac{1}{2}\|\mA x-b\|_{\mH}^2$ solves \eqref{linear_system} \footnote{In order to solve \eqref{linear_system} via minimizing \eqref{stochasticoptimization_problem}, one only needs local information of the stochastic function $f_{\mS}(x)$, for example, the stochastic gradient $\nabla f_{\mS}(x)$ without any explicit access to the function, its gradient, or the Hessian.}. At this end, Richt\'{a}rik and Tak\'{a}\v{c}, in \cite{richtarik2017stochastic} proposed a set of simple and easy-to-implement stochastic optimization algorithms----the {\em basic method}, the {\em parallel or minibatch method}, and an {\em accelerated method} to solve the stochastic reformulations of the linear system~\eqref{linear_system}. The {\em basic method} is the primary algorithm  whose iterative updates can be seen as SGD steps (as given in \eqref{sgd} \cite{robins_monro}) applied to solve the problem \eqref{stochasticoptimization_problem} with a fixed stepsize parameter.~Therefore, to solve \eqref{linear_system}, minimize the stochastic objective function $f_{\mS}(x)=\frac{1}{2}\|\mA x-b\|_{\mH}^2$, and the iterates of the {\em basic method} takes the following form:
 \begin{eqnarray}\label{basic_method}
 x_{k+1}=x_k-\omega \mB^{-1}\mA^\top \mS_{k} (\mS_{k}^\top \mA \mB^{-1} \mA^\top \mS_{k})^{\dagger} \mS_{k}^\top(\mA x_k-b),
\end{eqnarray}
where $\mS_k$ is sampled from the distribution $\cD$ in each iteration and $\omega>0$ is a fixed stepsize parameter. As mentioned before, the {\em parallel/minibatch method} is a natural extension of the {\em basic method} applied to a synchronized system of $\tau$ processors such that each one of them can perform an SGD step. At the end the user averages the total yield to perform the iterative update step. We mention the {\em parallel method} formally in Section \ref{parallel_method}. 

\subsection{Contribution}
In this paper, we solve \eqref{stochasticoptimization_problem} via a stochastic reformulation of the linear system \eqref{linear_system} in a distributed asynchronous set-up. We consider an ensemble with a central {\em master} server and say, $\tau$ independent {\em workers}, where the {\em master} obtains the gradients from the {\em workers} with a delay. Although, we follow the framework of the {\em basic method} proposed by Richt\'{a}rik and Tak\'{a}\v{c}  \cite{richtarik2017stochastic} to design our asynchronous SGD, our algorithm is closely related to Hogwild! of Recht et al. \cite{recht2011hogwild} and inspired by the delayed proximal gradient algorithm of Feyzmahdavian et al. \cite{feyzmahdavian2014delayed}. In a shared-memory model with $\tau$ independent {\em workers}, our iterative scheme updates the vector $x$ that is accessible to all workers. Each {\em worker} can contribute an update to the vector $x$, although, they can be of different processing speeds. Therefore, whenever a {\it worker} computes a stochastic gradient at $x$, it performs a SGD step at that point and communicates the update to the {\it master} processor. The {\em master} eventually experiences a delay and updates the final iterate by using a convex combination of its current instance of the vector $x$ and the SGD update that was communicated with a delay. We explain this process formally in Section \ref{asyn_sgd}.~The following are our main contributions in this paper:
\begin{itemize}
    \item Inspired by Richt\'{a}rik and Tak\'{a}\v{c}'s parallel basic method in \cite{richtarik2017stochastic}, we design an asynchronous parallel SGD to solve a consistent linear system. See Algorithm \ref{alg_2} in Section \ref{asyn_sgd}. 
    \item We propose a detailed convergence analysis of our algorithm in Section \ref{conv}. We compare the convergence rates of our asynchronous algorithm with the basic method and the synchronous parallel algorithm proposed by Richt\'{a}rik and Tak\'{a}\v{c} in  \cite{richtarik2017stochastic}. At this end, we consider different choices of the stepsize parameter~($\omega$), the damping factor~($\theta$), number of processors~($\tau$), and the delay factor~($\delta$) to analyze our results. We refer the readers to Section \ref{asyn_sgd} for details about this parameters. Moreover, our convergence analysis does not require any stronger assumption such as sparsity as in \cite{mania2017perturbed} and Leblond et al. \cite{asaga}. We refer the readers to Table \ref{table1}  for a quick overview of these results. 
    \item We compare the iteration complexity of the asynchronous and synchronous method.  We find that asymptotically as the number of processors, $\tau$, approaches to $\infty$, asynchronous approach has a better iteration complexity than its synchronous counterpart for certain linear systems. Moreover, for such linear systems, we also compute the minimum number of processors such that the asynchronous method has a better iteration complexity than the synchronous method, and find that this number can be as low as $2$ in some cases, even for highly ill-conditioned problems. We show them in Table \ref{table:case1} and Table \ref{table:case2}.
\end{itemize}

\subsection{Centralized algorithms---Related work}
As the digital data and computing power of the processors are increasing, in the past decade, there has been a strong focus on developing the parallel versions of stochastic algorithms. Based on the communication protocol, these parallel algorithms can be broadly classified as centralized~(that follows a {\it master-worker architecture}) and decentralized (for example, {\tt Allreduce} communication strategy, see \cite{layer-wise, framework}). In this paper we will follow the centralized set-up, where a {\em master} node coordinates with all the {\em worker} nodes. 
Depending on the update rule, the {\em centralized} algorithms can be further categorized into two categories---synchronous and asynchronous. In this scope, for completeness, we will quote a few representatives of each of those categories.~While Zinkevich et al.\ \cite{parallel_SGD} proposed and analyzed synchronous parallel SGD, Richt\'{a}rik and Tak\'{a}\v{c} in \cite{Richtarik2016} showed by parallelizing, randomized block coordinate descent methods can be accelerated. Yang \cite{NIPS2013_5114} proposed a distributed stochastic dual coordinate ascent algorithm in a {\em star-shaped distributed network}, and analyzed the trade-off between computation and communication.~In a similar sprit, Jaggi et al. \cite{NIPS2014_5599} proposed Communication-efficient distributed dual Coordinate Ascent or COCOA that uses an arbitrary dual optimization method on the local data on each computing node in parallel and reduces communication~(we refer to the references in \cite{NIPS2014_5599, takac2013mini} for distributed primal-dual methods; additionally, see \cite{necoara2013efficient} for application to distributed MPC).~Fercoq and Richt\'{a}rik in \cite{fercoq2015accelerated} proposed APPROX or Accelerated Parallel PROXimal method---a unison of three ideas, that is, acceleration, parallelization, and proximal method. In \cite{richtarik2016distributed}, Richt\'{a}rik and Tak\'{a}\v{c} proposed and analyzed a {\em hybrid} coordinate descent method known as HYDRA that partitions the coordinates over the nodes, independently from the other nodes, and applies updates to the selected coordinates~in parallel~(we also refer to \cite{Bradley2011ParallelCD, marevcek2015distributed} and HYDRA-2 \cite{fercoq2014fast} for more insights). In a similar line of work, Shamir et al. \cite{DANE} proposed a distributed approximate Newton-type method or DANE. Synchronous stochastic algorithms are well explored regarding their convergence rates, acceleration, and parallelization \cite{Shalev-Shwartz_SDCA, johnson_zhang}. However, they may suffer from the {\em memory locking}, that is, the processors or the computing nodes need to wait for the update from the slowest node. Hogwild! by Recht et al. \cite{recht2011hogwild} is one of the prime example of asynchronous stochastic algorithms and first one of its kind which do not use the {\em memory locking} protocol and as a result, the computing nodes can modify the parameters at the same time.~De Sa et al. \cite{de2015taming} proposed Buckwild! which is a low-precision asynchronous SGD. Additionally, they analyzed Hogwild! type algorithms with relaxed assumptions and analyzed asynchronous SGD algorithms for (non-convex) matrix-completion type problems (also see \cite{nguyen2018sgd}). Chaturapruek et al. \cite{chaturapruek2015asynchronous} showed that for convex problems, under similar conditions as regular SGD, asynchronous SGD achieves similar asymptotic convergence rate. However, as in \cite{chaturapruek2015asynchronous} the perturbed iterate analysis of Mania et al.\ \cite{mania2017perturbed} and Leblond et al. \cite{asaga} for proving the convergence of asynchronous SGD use rigorous sparsity assumption. 
Noel et al. \cite{noel2014dogwild} proposed Dogwild! that is distributed hogwild for CPU and GPU. In the advent of the deep neural networks, asynchronous parallel SGD type algorithms are highly deployed in practice. 
Recently, Lian et al.~\cite{Lian} 
showed that in their setting, proposed asynchronous parallel algorithms can achieve a linear speedup if the number of workers are bounded by the square root of the total number of iterations. In 2017, Zheng et al.~\cite{Zheng} proposed an algorithm called delay compensated asynchronous SGD (DC-ASGD) for training deep neural networks and to compensate the {\em delayed} gradient update by the local workers to the global model.~With experimental validity on deep neural networks, Zheng et al. claimed that their DC-ASGD outperforms both synchronous SGD and asynchronous SGD, and nearly approaches the performance of sequential SGD. Among the others, asynchronous algorithms by Aytekin et. al \cite{feyzmahdavian}, distributed SDCA by Ma et. al \cite{takac_icml}, distributed SVRG by Lee et. al, \cite{lee-svrg} and Zhao and Li \cite{li_svrg}, asynchronous parallel SAGA by Leblond et al. \cite{asaga}, proximal asynchronous SAGA by Pedregosa et al. \cite{pasaga}, are to name a few.~We refer the readers to \cite{de2017understanding,ben2019demystifying} for an in-depth understanding.

\paragraph*{Notation}
We provide a table of the most frequently used notation in this paper for convenience (See Appendix \ref{sec:notation}). 
Here we include some basic notations. We write the matrices
in bold uppercase letters and denote vectors and scalars by simple lowercase letters. We define the range space and null space of a matrix $\mA\in\R^{m\times n}$ as ${\rm Im}(\mA)\eqdef\{y\in\R^m : y=\mA x\}$ and $N(\mA)\eqdef\{x\in\R^n : \mA x=0\}$, respectively. We further define the Euclidean inner product as $\langle \cdot,\cdot \rangle$ and for a symmetric positive definite matrix $\mB$, we denote $\langle \cdot,\cdot \rangle_{\mB}$ as the $\mB$-inner product and define $\|x\|_{\mB}=\sqrt{x^\top\mB x}$ as the (semi)-norm induced by it. 

\paragraph*{Organization}
The paper is organized as follows. In Section 2, we review some key results related to the stochastic reformulation of linear systems and describe the synchronous {\em parallel method} by Richt\'{a}rik and Tak\'{a}\v{c}, in \cite{richtarik2017stochastic}. Next in Section 3, we present our asynchronous parallel SGD. We compare the convergence rates of the asynchronous SGD with the synchronous {\em parallel method} in Section \ref{conv}.

\begin{table}
\centering
\begin{tabular}{|c |c |c |c |c |c |c |c |}     
 \hline
Algorithm &Quantity & Case & $\omega$  & $\theta$ & $\tau$ & Complexity & Reference\\
 \hline
& &    & $1$    & $\theta_1$ &   $\tau$ & $\tfrac{\xi_a(1,\tau)}{\lambda_{\min}^+}$ & \\
APSGD &$\mathbb{E}[\|x_t-x_\star\|_{\mB}^2]$ & $\omega^\star\leq2 $& $1$    & $\theta_1$ &   $\infty$ & $\tfrac{3}{4\lambda^+_{\rm min}}$ & This paper\\
& && $\omega^\star$    & $\theta_{\omega^\star}$ &   $\tau$ & $\tfrac{\xi_a(\omega^\star,\tau)}{\lambda_{\min}^+}$&\\
 & && $\omega^\star$    & $\theta_{\omega^\star}$ &   $\infty$ & $\tfrac{3\lambda_{\min}^++\lambda_{\rm max}}{4\lambda_{\min}^+}$ &\\
 \hline
 && & $1$    & $\theta_1$ &   $\tau$ & $\tfrac{\xi_a(1,\tau)}{\lambda_{\min}^+}$ & \\
APSGD & $\mathbb{E}[\|x_t-x_\star\|_{\mB}^2]$ &$\omega^\star\geq2 $   & $1$    & $\theta_1$ &   $\infty$ & $\tfrac{3}{4\lambda^+_{\rm min}}$ &This paper\\
 &&& $2$    & $\theta_2$ &   $\tau$ &  $\tfrac{\xi_a(2,\tau)}{\lambda_{\min}^+}$&\\
 &&& $2$    & $\theta_2$ &   $\infty$ & $\tfrac{1+2\lambda^+_{\rm min}}{4\lambda^+_{\rm min}}$ &\\
 \hline
&&   & $1$    & - &   $\tau$ & $\tfrac{1}{(2-\xi_s(\tau))\lambda_{\rm min}^+}$ & \cite{richtarik2017stochastic} \\
Parallel SGD &$\mathbb{E}[\|x_t-x_\star\|_{\mB}^2]$ &$\omega\in(0,2/\xi_s(\tau))$ & $1/\xi_{\tau}$& -&$\tau$ & $\tfrac{\xi_S(\tau)}{\lambda_{\min}^+}$ &\cite{richtarik2017stochastic}\\
&&& $1/\lambda_{\rm max}$& -&$\infty$ &  $\tfrac{\lambda_{\rm max}}{\lambda_{\rm min}^+}$& \cite{richtarik2017stochastic}\\
\hline
\end{tabular}
\caption{\small{Iteration complexities of asynchronous parallel SGD~(APSGD) and parallel SGD or parallel basic method. For asynchronous SGD, we define, $\xi_a(1,\tau)  \eqdef\frac{3}{4}+\frac{1+\sqrt{1+2c\tau(1-\lambda_{\min}^+)}}{4c\tau}$, $\xi_a(\omega^\star,\tau)\eqdef \frac{3\lambda_{\min}^++\lambda_{\rm max}}{4}+\frac{\sqrt{1+c\tau(2-k)}+2(c\tau+1+c\tau(1-k)\lambda_{\min}^+)}{2c\tau\sqrt{1+c\tau(2-k)}}$ and $\xi_a(2,\tau) \eqdef\frac{1}{4}+\frac{\lambda_{\min}^+}{2}+\frac{1+\sqrt{c\tau+1}}{2c\tau}$, where $k=\lambda_{\min}^+ +\lambda_{\rm max}$ and $c\ge 1$. We also define, $\theta_1 \eqdef\frac{\sqrt{1+2c\tau(1-\lambda_{\min}^+)}-1}{c\tau(1-\lambda_{\min}^+)} $, $\theta_2\eqdef \tfrac{\sqrt{c\tau+1}-1}{c\tau}$, and $\theta_{\omega^\star} \eqdef \tfrac{k(\sqrt{1+c\tau(2-k)}-1)}{c\tau(2-k)}$.~For parallel SGD, we define $\xi_S(\tau)\eqdef \frac{1}{\tau}+\left(1-\frac{1}{\tau}\right)\lambda_{\rm max}$.}}\label{table1}
\end{table}

\section{Stochastic reformulation of a linear system: A few key results}\label{sec:2}
Based on Section \ref{sec:stochastic_reformulation}, we are now set to quote some results without their proofs which follow directly from \cite{richtarik2017stochastic, gower2015stochastic} and are used to establish our main results. 
Let $${\mZ = \mZ_\mS \eqdef \mA^\top \mS ( \mS^\top \mA \mB^{-1} \mA^\top \mS)^{\dagger} \mS^\top \mA}$$
and $\Exp[\mZ] \eqdef \Exp_{\mS\sim \cD}[\mZ]$ be such that $\cD$ is a user-defined distribution and $\mS$ be a random matrix drawn from $\cD$. Let $\mB$ be a $n\times n$ symmetric positive definite matrix. Define
$$ {\mW\eqdef\mB^{-\frac{1}{2}}\Exp[{\mZ}]\mB^{-\frac{1}{2}}}$$
and let $\mW=\mU\Lambda\mU^\top$ be a eigenvalue decomposition of $\mW$, where $\mU^\top\mU=\mU\mU^\top=\mI$ and $\Lambda$ is a diagonal matrix with eigenvalues $0\le\lambda_i\le1$ arranged in a non-increasing order. Additionally, we note that $\lambda_{\rm max}=\lambda_1$ is the largest and $\lambda_{\rm min}^+$ is the smallest non-zero eigenvalue of $\mW.$ We start by defining the {\em exactness} assumption. 

\begin{assumption}\cite{richtarik2017stochastic}\label{assump1}
Let $\cX={\arg\min_{x\in\R^n} f(x)=\{x\;:f(x)=0\}=\{x\;:\nabla f(x)=0\}}.$ Then $\cX=\cL.$
\end{assumption}
By $x_\star=\Pi^{\mB}_{\cL}(x_0)$ we denote $x_\star$ to be the projection of the initial iterate $x_0$ onto the set $\cL$ in $\mB$-norm and quote the following results. 
\begin{remark}\label{gradient_hess}
For $x_\star\in\cL$, the gradient and Hessian of $f$ are $\nabla f(x)=\mB^{-1}\E{\mZ}(x-x_\star)$ and $\nabla^2 f(x)=\mB^{-1}\E{\mZ},$ respectively.
\end{remark}
\begin{lemma}[Lemma 4.7 in \cite{richtarik2017stochastic}] \label{eq:09s9hsoiuis907} For all $x\in \R^n$, $x_\star\in \cL$ and for a given $\mS$ we have
\begin{equation}   \|x - x_\star - \omega \nabla f_{\mS}(x)\|_{\mB}^2 = \|(\mI  - \omega \mB^{-1} \mZ)(x-x_\star) \|_{\mB}^2 = \|x-x_\star\|_{\mB}^2 - 2 \omega(2-\omega)f_{\mS}(x) .\end{equation}
\end{lemma}

\begin{lemma}[Lemma~4.2 in \cite{richtarik2017stochastic}] \label{lem:osohhd9u93}
For all $x \in \R^n$ and  $x_\star= \Pi_{\cL}^\mB(x)$ we have $\frac{\lambda^+_{\min}}{2}\|x-x_\star\|_\mB^2\leq f(x)\leq\frac{\lambda_{\rm max}}{2}\|x-x_\star\|_\mB^2.$
\end{lemma}

\begin{lemma}[Lemma~4.5 in \cite{richtarik2017stochastic}] \label{lem:hbs6763vs6} 
Consider any $x\in \R^n$ and $x_\star = \Pi^{\mB}_{\cL}(x)$. If $\lambda_i = 0$ then we have $u_i^\top \mB^{1/2} (x-x_\star)=0$.
\end{lemma}

\subsection{Parallel method}\label{parallel_method}
In this section, we describe the {\it parallel basic method} proposed by Richt\'{a}rik and Tak\'{a}\v{c} in \cite{richtarik2017stochastic}. This is also referred to as the {\em minibatch} method. Let there be $\tau$ processors working independently. Starting from a given iterate $x_k$, the {\em parallel} or {\em minibatch method} performs one step of the {\em basic method} independently on each of the $\tau$ processors and finally averages the results. This leads to the update rule of {\em parallel basic method} (see \cite{richtarik2017stochastic}) as follows:
\begin{equation}\label{basic_method_parallel} 
 x_{k+1}=\frac{1}{\tau} \sum_{i=1}^\tau z_{k+1,i},
\end{equation}
where 
$z_{k+1,i}=x_k-\omega \mB^{-1}\mA^\top \mS_{ki} (\mS_{ki}^\top \mA \mB^{-1} \mA^\top \mS_{ki})^{\dagger} \mS_{ki}^\top(\mA x_k-b). $
Note that, in each iterate an independent sample $\mS_{ki}\sim \cD$ drawn afresh for each of the $\tau$ processors and $\omega>0$ is a fixed stepsize parameter. One can see the {\em parallel basic method} as a {\em synchronous parallel method}. 
In fact, the {\em parallel basic method} enjoys a linear convergence rate under the {\em exactness} assumption. 
\begin{theorem}[Convergence of parallel basic method~\cite{richtarik2017stochastic}]\label{minibatch}
Let the exactness assumption hold and $x_\star= \Pi_{\cL}^\mB(x_0)$. Let $\{x_k\}_{k\ge 0}$ be the sequence of random iterates produced by the parallel method (see \eqref{basic_method_parallel}) where the stepsize $\omega\in(0,2/\xi_s(\tau))$, such that $\xi_s(\tau)=\frac{1}{\tau}+(1-\frac{1}{\tau})\lambda_{\rm max}.$ Then 
\begin{eqnarray}\label{parallel_conv}
\Exp[\|x_{k+1}-x_\star\|_{\mB}^2] \leq \rho_{s}(\omega,\tau)^k\frac{\lambda_{\rm max}}{2}\|x_{0}-x_\star\|_{\mB}^2,
\end{eqnarray}
where
\begin{eqnarray}\label{parallel_rho}
\rho_{s}(\omega,\tau)=1-\omega(2-\omega\xi_s(\tau))\lambda_{\rm min}^+.
\end{eqnarray}
\end{theorem}
 \cite{richtarik2017stochastic} further showed that $\rho_{s}(\omega,\tau)$ is minimized for $\omega(\tau)=1/\xi_s(\tau)$ and the optimal $\rho_{s_{\rm opt}}(\omega(\tau),\tau)$ is 
\begin{eqnarray}\label{optimal_rho}
\rho_{s_{\rm opt}}(\omega(\tau),\tau)=1-\frac{\lambda_{\rm min}^+}{\frac{1}{\tau}+\left(1-\frac{1}{\tau}\right)\lambda_{\rm max}}.
\end{eqnarray}
Let $\chi_{s_{opt}}(\tau)$ be the best iteration complexity of the synchronous parallel method. Then as in \cite{richtarik2017stochastic}, we find:
\begin{equation}\label{eq:ajd}
    \chi_{s_{opt}}(\tau)= \mathcal{O}\left(\frac{\frac{1}{\tau}+\left(1-\frac{1}{\tau}\right)\lambda_{\rm max}}{\lambda_{\rm min}^+}\right)=\mathcal{O}\left(\frac{\xi_s(\tau)}{\lambda_{\rm min}^+}\right).
\end{equation}
which is achieved at $\omega=\frac{1}{\xi_s(\tau)}$.
\begin{remark}
The best strong convergence rate for the basic method is achieved when the stepsize parameter $\omega=1$. We will define the strong convergence in Section \ref{strong_convergence}.
\end{remark}

\section{Asynchronous parallel SGD}\label{asyn_sgd}
Let there be $\tau$ independent processors or {\em workers} and a central server or the {\em master} (as described in Section \ref{intro}) and we perform the iterative updates in an {\it asynchronous} manner. That is, whenever a {\it worker} computes a stochastic gradient at a given point, it performs a SGD step at that point and communicates the update to the {\it master} processor. After that, the {\em master} generates a new update by using a convex combination of its current update and the update reported to it by the {\em worker}, and communicates back the resulting update to the {\em worker} to perform the next SGD step, and this process continues. Regardless of their processing speeds, whenever a {\em worker} communicates its latest update to the {\em master}, the master makes a convex combination of the latest iterate with it and assigns it to the worker to perform an SGD step. Therefore, no {\em worker} stays idle. Each of them performs and communicates the task they are assigned to the {\em master}, albeit in an asynchronous way. This iterative protocol was introduced in \cite{feyzmahdavian2014delayed} and our proposed update rule is inspired by it. 

Let $t$ denote the iteration count with respect to the {\em master}'s frame. Let $\delta\geq 0$ be a factor representing {\em delay} between the iterates of the {\em master} and the {\em workers}, which indeed varies in each iteration. Let $\delta_a$ be the maximum delay throughout the execution of the algorithm. Note that the delay $\delta$ is functions of $t$. But for brevity, we will write it as $\delta$ in this paper. Let $\theta\in[0,1]$ represent a damping factor. Choose $x_0\in \R^n$ and consider an iterative method defined for $t\geq 1$ as:
\begin{eqnarray}
y_{t}&=&x_{t-\delta}-\omega\nabla f_{\mS_{t-\delta}}(x_{t-\delta}),\label{eq:worker}\\
x_{t+1} &=& (1-\theta)x_t+\theta y_{t}.\label{eq:master}
\end{eqnarray}
Recall from Remark \ref{gradient_hess} that  $\nabla f_{\mS}(x) = \mB^{-1} \mZ(x-x_\star)$, where $x_\star$ is any solution of the system $\mA x = b$. Therefore, we can rewrite the update rule as:
\begin{equation} \label{eq:jd8d8bs9jhisoiP} 
\boxed{x_{t+1} - x_\star = (1-\theta)(x_t-x_\star) + \theta(\mI -\omega \mB^{-1}\mZ_{t-\delta})(x_{t-\delta}-x_\star),}
\end{equation}
where we denote $\mZ_{t} = \mZ_{\mS_t}$ and $\mS_t$ is sampled independently from $\cD$. 

\begin{algorithm}
\SetAlgoLined
    \SetKwInOut{Master}{Master}
    \SetKwInOut{Workers}{Workers}
    \SetKwInOut{System}{System}
    \SetKwInOut{repeat}{repeat}
    \System{$1$ Master and $K$ workers, $k=1,.....,K$;}
    \nl\Repeat{convergence}{
        \nl\Master\
        \nl Receive $y_t$ from a worker $k$\;
        \nl Update current iterate $x_t$ by using equation \eqref{eq:master}\;
     \nl Send updated iterate $x_{t+1}$ to worker $k$\;
        \nl\Workers\
       \nl Receive iterate $x_t$ from master\;
      \nl Compute $y_{t+\delta}$ via equation \eqref{eq:worker}\;
       \nl Send $y_{t+\delta}$ to the master\;}   
    \caption{Distributed Asynchronous SGD in master-worker architecture (\textit{time t in master's frame})}\label{alg_2}
\end{algorithm}

 \section{Convergence analysis: Comparison with the synchronous parallel method}\label{conv} 
Denote a {\it unit time interval} as the time in which the slowest {\em worker} performs an update. By slowest, we mean having the maximum delay between an iterate sent to the worker by the master, to its SGD update being considered by the master. The {\em slowness} can be attributed due to less processing power, or/and due to processing a ``tougher" job in terms of time complexity, or due to slower communication between this worker and the master. For the sake of simplicity, we assume the number of updates by the asynchronous parallel SGD in a unit-time interval to be a constant throughout the execution of the algorithm. Thus, in a unit-time interval, we assume exactly $\delta_a$ updates by the master. Let $S_i$ be the number of updates the $i^{th}$ processor performs in this time interval. Thus, $\delta_a=\sum_iS_i.$
We note that in this unit time interval, the synchronous parallel  method performs one step of the basic method independently $\tau$ times and averages the results, where $\delta_a\geq\tau$. The last inequality is obvious because of the fundamental structure of the asynchronous method. Thus, the asynchronous algorithm leverages the idle time of the faster worker to perform more updates in a unit-time interval in comparison to its synchronous counterpart, but in an asynchronous manner. 
Let $c\geq1$ be such that $\delta_a=c\tau$.
~Let $\rho_a(\theta,\omega)$ be the rate of convergence of the asynchronous parallel SGD in a unit-time interval, with step-size $\omega$ and damping parameter $\theta$.

\subsection{Key results used for convergence}
In this section we propose a few key results that are necessary for our convergence analysis in Section \ref{strong_convergence}. As in \cite{richtarik2017stochastic}, we do not formally provide the {\it weak} convergence analysis \footnote{A sequence of random vectors $\{x_k\}_{k\ge 0}$ (the iterates) converge to $x_\star$ weakly if $\|\E{x_k-x_\star}\|_{\mB}^2\to 0$ as $k\to\infty$.}
of our asynchronous parallel method (see \eqref{eq:jd8d8bs9jhisoiP}). The convergence results in Section \ref{strong_convergence} is {\it strong} convergence~(and that is how the algorithm is expected to behave in practice). Moreover, the weak convergence can be directly implied from the strong convergence result. First, in Theorem \ref{weak_recursion} we establish a recurrence relation that can lead to the weak convergence of our algorithm. 
\begin{theorem}\label{weak_recursion}
Denote $P_t=\mU^\top\mB^{1/2}\mathbb{E}[x_{t}-x_{\star}]$. Then from \eqref{eq:jd8d8bs9jhisoiP} we obtain the following:
\begin{eqnarray}\label{week_rec}
 \mU^\top\mB^{1/2}\mathbb{E}[x_{t+1}-x_{\star}]=(1-\theta)\mU^\top\mB^{1/2}\mathbb{E}[x_{t}-x_{\star}]\nonumber\\
+\theta(\mI-\omega\Lambda)\mU^\top\mB^{1/2}\mathbb{E}[x_{t-\delta}-x_{\star}].
\end{eqnarray}
\end{theorem}
\begin{proof}
See Appendix \ref{sec:proofs}.
\end{proof}

\begin{remark}
One can split \eqref{week_rec} coordinate-wise. Thus for each $1\le i\le n$, \eqref{week_rec} can be written as
\begin{eqnarray}\label{week_rec_vec}
P_{t+1}^i=(1-\theta)P^i_{t}+\theta(1-\omega\lambda^i)P^i_{t-\delta},
\end{eqnarray}
where $P^i_{t}$ is the $i^{\rm th}$ coordinate of $P_{t}$. 
Therefore, we are left to understand the convergence of the following recurrence relation:
\begin{equation}\label{rec}
    p_{t+1}=a\ p_t+b\ p_{t-\delta},
\end{equation}
where $a=1-\theta>0, b=\theta(1-\omega\lambda_i)>0,$ and $a+b<1.$ 
\end{remark}
The state transition matrix for the recurrence relation in \eqref{rec} is
\begin{eqnarray}\label{state_transition}
\begin{pmatrix}
p_t\\\vdots\\p_{t-\delta}
\end{pmatrix}=
\mA\begin{pmatrix}
p_{t-1}\\\vdots\\p_{t-\delta-1}
\end{pmatrix}, 
\end{eqnarray}
where $\mA = \begin{pmatrix}
&[a\;\;\mathbf{0}_{n-2}^\top] &b\\
&I_{n-1} &\mathbf{0}_{n-1}\end{pmatrix}$ and $\mathbf{0}_n$ is a vector in $\R^n$ with all zeros.~One can analyze the convergence of \eqref{week_rec} by analyzing the spectral radius, $\rho(\mA)$, of the state transition matrix $\mA$.~Note that, the characteristic equation of the matrix $\mA$ is:
\begin{equation}\label{pol}
 \gamma^{\delta+1}-a\gamma^{\delta}-b=0.   
\end{equation}
Calculating the spectral radius of $\mA$ is equivalent to finding the magnitude of the highest root of the polynomial in \eqref{pol}.~We quote two classic results: one is on the bounds of the root of a polynomial and the second one is on the spectral radius of a non-negative matrix. At the core they are the same results and complement each other. 

\begin{theorem}\label{Cauchy}[Cauchy~\cite{polynomials}]
Let $f(x) = x^n- \sum_{i=1}^n b_ix^{n-i}$ be a polynomial, where all the coefficients $b_i$'s are non-negative and at least one of them is nonzero. The
polynomial $f(x)$ has a unique (simple) positive root $p$ and the absolute values of the other roots do not exceed $p$.
\end{theorem}

\begin{theorem}\label{PF Theorem}[Perron-Frobenius~\cite{PF}]
Let $\mM=(m_{ij})_{n\times n}$ be a square irreducible non-negative matrix, that is $m_{ij}\geq0$ for $1\le i,j\le n$. Then there exists a positive real number $r$ such that $r$ is a simple eigenvalue of the matrix $\mM$ and the magnitude of all other eigenvalues of $\mM$ is strictly smaller than $r$. Therefore $\rho(\mM)=r$. Moreover,  $r$ satisfies 
$ \displaystyle {\rm \min}_i\sum_jm_{ij}\le r\le {\rm max}_i\sum_jm_{ij}.
$
\end{theorem}

Theorem \ref{Cauchy} was proposed by Cauchy and Theorem \ref{PF Theorem} is the famous Perron-Frobenius Theorem on non-negative matrices. In our case, Cauchy's theorem is directly related to the Perron-Frobenius theorem. The positive real number $r$ in Theorem \ref{PF Theorem} is same as $p$ in Theorem \ref{Cauchy} and is known as the {\it Perron root} or {\it Perron-Frobenius eigenvalue}. It is easy to check that the associated digraph for the matrix $\mA$ is strongly connected and hence $\mA$ is an irreducible non-negative matrix. By Theorem \ref{Cauchy} (or Theorem \ref{PF Theorem}) we can conclude the spectral radius of $\mA$ is the unique positive real root of \eqref{pol}. 

\subsection{Strong convergence}\label{strong_convergence}
To demonstrate the convergence of our asynchronous parallel method, we follow the convention used by \cite{richtarik2017stochastic} and refer to this convergence analysis as {\it strong} convergence analysis. We recall that Richt\'{a}rik and Tak\'{a}\v{c} in \cite{richtarik2017stochastic} defined the strong convergence as follows: A sequence of random vectors $\{x_k\}_{k\ge 0}$ (the iterates) converge to $x_\star$ strongly if $\E{\|x_k-x_\star\|_{\mB}^2}\to 0$ as $k\to\infty.$ We start with the following Lemma. 

\begin{lemma} \label{lem:bs983bv78dv} 
Assume that $x_0 \in {\rm Im}(\mB^{-1}\mA^\top)$. Then
$x_t \in {\rm Im}(\mB^{-1}\mA^\top)$ for all $t$. It follows that  $\Pi^{\mB}_{\cL}(x_t) = \Pi^{\mB}_{\cL}(x_0) $ for all $t$.
\end{lemma}

 \begin{proof}
See Appendix \ref{sec:proofs}.
\end{proof}

\subsubsection{Two scenarios based on the relative position of \texorpdfstring{$\omega^\star$}{Lg} and 2}
Let ${\omega^\star=\frac{2}{(\lambda^+_{\min}+\lambda_{\max})}}$. According to the problem, there may be two scenarios based on the relative position of $\omega^\star$ with respect to 2 as follows:

\textbf{Case 1}: $\omega^\star\leq 2 \quad \text{i.e.} \quad \lambda^+_{\min}+\lambda_{\max} \geq 1 $.\newline

\begin{tikzpicture}
    \draw (0,0) -- (3,0);
    \foreach \i in {0,1,2} 
      \draw (\i,0.1) -- + (0,-0.2) node [below] {$\i$};
    \fill[black] (1.5,0) circle (0.6 mm) node[below] {$\omega^\star$};
\end{tikzpicture}

\textbf{Case 2}: $\omega^\star\geq 2 \quad \text{i.e.}\quad \lambda^+_{\min}+\lambda_{\max} \leq 1 $.\newline

\begin{tikzpicture}
    \draw (0,0) -- (3,0);
    \foreach \i in {0,1,2} 
      \draw (\i,0.1) -- + (0,-0.2) node[below] {$\i$};
    \fill[black] (3,0) circle (0.6 mm) node[below] {$\omega^\star$};
\end{tikzpicture}

\begin{definition}
Define $\alpha(\omega) \eqdef \max_{i:\lambda_i>0}|1-\omega\lambda_i|$. It is easy to see that 
\begin{align*}
\alpha(\omega) = 
     \begin{cases}
       1-\omega\lambda^+_{\min}, &\quad\text{if }0\le\omega\le\omega_{\star}\\
       \omega\lambda_{\max}-1, &\quad\text{if }\omega\ge\omega_{\star}.
     \end{cases}
\end{align*}
\end{definition}

\subsubsection{Recurrence relation}
The following theorem establishes a recurrence relation that we need to analyze the strong convergence of our asynchronous parallel method. 
\begin{theorem}[Recurrence]\label{recurrence} Assume exactness. Assume that $x_0,\dots,x_{\delta_a} \in{\rm  Im}(\mB^{-1}\mA^\top)$ and let $x_\star = \Pi^{\mB}_{\cL}(x_0)$. Let $\{x_t\}$ be the sequence of random iterates produced by the asynchronous method (via \eqref{eq:jd8d8bs9jhisoiP}) with delay $\delta\geq 0$, stepsize $ \omega \geq 0$, and damping parameter $\theta \in [0,1]$. Let $r_t = \mB^{1/2}(x_t-x_\star)$ and $\alpha(\omega) \eqdef \max_{i:\lambda_i>0}|1-\omega\lambda_i|$. Then
\begin{eqnarray}\label{recursion_strong}
\boxed{\Exp[\|r_{t+1}\|^2] \leq K_1(\theta,\omega) \Exp[\|r_t\|^2]
+ K_2(\theta,\omega) \Exp[\|r_{t-\delta}\|^2],}
\end{eqnarray}

and we have the following estimates of $K_1(\theta,\omega)$ and $K_2(\theta,\omega)$:\\
\textbf{Case 1:}
\begin{enumerate}[label=(\roman*)]
\item For $\omega\leq\omega^\star\leq 2$ and $\alpha(\omega)=1-\omega\lambda^+_{\min}$:\newline

\begin{tikzpicture}
    \draw (0,0) -- (4,0);
    \fill[black] (0,0) circle (0.6 mm) node[below] {$0$};
    \fill[black] (4,0) circle (0.6 mm) node[below] {$2$};
    \fill[black] (3,0) circle (0.6 mm) node[below] {$\omega^\star$};
    \fill[black] (2.3,0) circle (0.6 mm) node[below] {$\omega$};
\end{tikzpicture}

\begin{equation}\label{scenario_a_1}
K_1(\theta,\omega) \eqdef  (1-\theta)(1-\theta + \theta\alpha(\omega)), \qquad K_2(\theta,\omega) \eqdef \theta (\theta (1-\omega(2-\omega)\lambda_{\min}^+) + (1-\theta)\alpha(\omega) ).
\end{equation} 

\item  For $\omega^\star\leq\omega\leq 2$ and $\alpha(\omega)=\omega\lambda_{\max}-1$:\newline

\begin{tikzpicture}
    \draw (0,0) -- (4,0);
    \fill[black] (0,0) circle (0.6 mm) node[below] {$0$};
    \fill[black] (4,0) circle (0.6 mm) node[below] {$2$};
     \fill[black] (2,0) circle (0.6 mm) node[below] {$1$};
    \fill[black] (2.5,0) circle (0.6 mm) node[below] {$\omega^\star$};
    \fill[black] (3,0) circle (0.6 mm) node[below] {$\omega$};
\end{tikzpicture}

\[K_1(\theta,\omega) \eqdef  (1-\theta)(1-\theta + \theta\alpha(\omega)), \qquad K_2(\theta,\omega) \eqdef \theta (\theta (1-\omega(2-\omega)\lambda_{\min}^+) + (1-\theta)\alpha(\omega) ).\]

\item For $\omega\geq 2$ and $\alpha(\omega)=\omega\lambda_{\max}-1$:\newline

\begin{tikzpicture}
   \draw (0,0) -- (4,0);
    \fill[black] (0,0) circle (0.6 mm) node[below] {$0$};
    \fill[black] (3,0) circle (0.6 mm) node[below] {$2$};
     \fill[black] (1.5,0) circle (0.6 mm) node[below] {$1$};
    \fill[black] (2.5,0) circle (0.6 mm) node[below] {$\omega^\star$};
    \fill[black] (4,0) circle (0.6 mm) node[below] {$\omega$};
\end{tikzpicture}

\[K_1(\theta,\omega) \eqdef  (1-\theta)(1-\theta + \theta\alpha(\omega)), \qquad K_2(\theta,\omega) \eqdef \theta (\theta (1-\omega(2-\omega)\lambda_{\max}) + (1-\theta)\alpha(\omega) ).\]

\end{enumerate}

\noindent\textbf{Case 2:}
\begin{enumerate}[label=(\roman*)]
\item \hspace{1em} For $\omega\leq2\leq\omega^\star$ and $\alpha(\omega)=1-\omega\lambda^+_{\min}$:\newline

\begin{tikzpicture}
    \draw (0,0) -- (4,0);
    \fill[black] (0,0) circle (0.6 mm) node[below] {$0$};
    \fill[black] (3,0) circle (0.6 mm) node[below] {$2$};
    \fill[black] (4,0) circle (0.6 mm) node[below] {$\omega^\star$};
    \fill[black] (1.5,0) circle (0.6 mm) node[below] {$\omega$};
\end{tikzpicture}

\begin{equation}\label{scenario_b_1}
K_1(\theta,\omega) \eqdef  (1-\theta)(1-\theta + \theta\alpha(\omega)), \qquad K_2(\theta,\omega) \eqdef \theta (\theta (1-\omega(2-\omega)\lambda_{\min}^+) + (1-\theta)\alpha(\omega) ).
\end{equation}

\item For $2\leq\omega\leq\omega^\star$ and $\alpha(\omega)=1-\omega\lambda^+_{\min}$:\newline

\begin{tikzpicture}
    \draw (0,0) -- (4,0);
    \fill[black] (0,0) circle (0.6 mm) node[below] {$0$};
    \fill[black] (2,0) circle (0.6 mm) node[below] {$2$};
    \fill[black] (4,0) circle (0.6 mm) node[below] {$\omega^\star$};
    \fill[black] (3,0) circle (0.6 mm) node[below] {$\omega$};
\end{tikzpicture}

\begin{equation}\label{scenario_b_2}
K_1(\theta,\omega) \eqdef  (1-\theta)(1-\theta + \theta\alpha(\omega)), \qquad K_2(\theta,\omega) \eqdef \theta (\theta (1-\omega(2-\omega)\lambda_{\max}) + (1-\theta)\alpha(\omega) ).
\end{equation}

\item For $\omega\geq\omega^\star$ and $\alpha(\omega)=\omega\lambda_{\max}-1$:\newline

\begin{tikzpicture}
    \draw (0,0) -- (4,0);
    \fill[black] (0,0) circle (0.6 mm) node[below] {$0$};
    \fill[black] (2,0) circle (0.6 mm) node[below] {$2$};
    \fill[black] (3,0) circle (0.6 mm) node[below] {$\omega^\star$};
    \fill[black] (4,0) circle (0.6 mm) node[below] {$\omega$};
\end{tikzpicture}

\[K_1(\theta,\omega) \eqdef  (1-\theta)(1-\theta + \theta\alpha(\omega)), \qquad K_2(\theta,\omega) \eqdef \theta (\theta (1-\omega(2-\omega)\lambda_{\max}) + (1-\theta)\alpha(\omega) ).\]
\end{enumerate}

\end{theorem}

\begin{proof}
See Appendix \ref{sec:proofs}.
\end{proof}


\begin{remark}
To analyze the convergence of the recurrence relation in~\eqref{recursion_strong} we can replace the inequality in~\eqref{recursion_strong} with equality and write
\begin{align}\label{rec3}
      q_{t+1}= K_1(\theta,\omega) q_t+ K_2(\theta,\omega) q_{t-\delta},
\end{align}
where we initialize the process by setting $q_t=\mathbb{E}\big[\|x_t-x_\star\|_\mB^2\big]$ for $t \in \{0,\dots,\tau\}$.
It can be easily seen by using induction that $\mathbb{E}\big[\|x_t-x_\star\|_{\mB}^2\big]\leq q_t$ for all $t$, and hence we can claim that the rate of convergence of $\mathbb{E}\big[\|x_t-x_\star\|_{\mB}^2\big]$ will not be slower than $q_t$. 
\end{remark}

\begin{remark}
Note that \eqref{rec3} is similar to \eqref{rec}. Therefore, one can obtain a recurrence relation that involves a state transition matrix similar to that in \eqref{state_transition}. To analyze the strong convergence of our method and to compare it against the parallel method, we examine the characteristic equation of the state transition matrix in \eqref{rec3}. 
\end{remark}

\subsubsection{Characteristic polynomial}\label{characteristic_polynomial}
The characteristic polynomial of the state transition matrix in recurrence \eqref{rec3}  is
 \begin{equation}\label{characteristic_poly}
   p_{\delta}(\gamma) \eqdef \gamma^{\delta+1}  - K_1(\theta,\omega) \gamma^\delta - K_2(\theta,\omega). 
 \end{equation}
 For convenience, we denote $K_1 =K_1(\theta,\omega)$ and $K_2 =K_2(\theta,\omega)$, from now on. By using Theorem \ref{Cauchy} (or Theorem \ref{PF Theorem}) we can conclude that the positive root of \eqref{characteristic_poly} is the spectral radius of \eqref{characteristic_poly}, and therefore, we can obtain the convergence rate of \eqref{rec3}. 
This motivates us to propose the following theorem. 
  \begin{theorem}\label{conv_strong}
 Denote $\varrho_{A}(\theta,\omega,\delta)$ as the spectral radius of \eqref{rec3}. Then $\varrho_{A}(\theta,\omega,\delta)$ is the convergence factor of \eqref{rec3}. 
 \end{theorem}
 \begin{remark}
 We see that the spectral radius is smaller than 1 only when $K_1+K_2<1$, as $p_{\delta}(1)$ should be positive for the root to be smaller than 1. If $K_1+K_2\geq1$, then either the spectral radius is not a good bound for the one-step rate of convergence or the algorithm itself does not converge. 
 \end{remark}
 \begin{lemma}\label{lem:jsh}
 For all the updates, the delay factor $\delta\leq\delta_a$, we have for all $\delta$
 \begin{equation}
 \varrho_{A}(\theta,\omega,\delta)\leq\varrho_{A}(\theta,\omega,\delta_a).
 \end{equation}
 \end{lemma}
 
 \begin{proof}
See Appendix \ref{sec:proofs}.
\end{proof}

 \begin{lemma}\label{lem:roc}
 Let $\delta_i$ denote the delay factor for the $i^{th}$ update in a unit time interval. Then, we have the following relation for unit time rate of convergence $\rho_a(\theta,\omega)$
 \begin{equation}\label{rate_of_convergence}
     \rho_a(\theta,\omega)\leq\prod_{i=1}^{\delta_a}\varrho(\theta,\omega,\delta_i)\leq\varrho(\theta,\omega,\delta_a)^{\delta_a}
 \end{equation}
 \end{lemma}
 
 \begin{proof}
See Appendix \ref{sec:proofs}.
\end{proof}

\subsubsection{Bounding polynomial for the characteristic polynomial}
 It is hard to find a closed form analytic expression for the unique positive root (or, the spectral radius) of the polynomial. Therefore, we find a polynomial that bounds $p_{\delta}(\gamma)$ in the interval $[0,1]$. This will provide us with a bound for the spectral radius and we can comment on the convergence factor of the asynchronous parallel method.  
\begin{lemma}\label{lem:kjd}
The polynomial  
\begin{equation}
 g_{\delta}(\gamma) \eqdef \left(1+\frac{1}{\delta}-K_1\right)\gamma^\delta  - \left(K_2+\frac{1}{\delta}\right)
 \end{equation}
bounds the characteristic polynomial $p_{\delta}(\gamma)$ from below on $[0,1]$ and its root
\begin{equation}\label{upper_bound}
    u(\theta,\omega,\delta)=\left(\frac{K_2+\frac{1}{\delta}}{1-K_1+\frac{1}{\delta}}\right)^{\frac{1}{\delta}}
\end{equation}
  is an upper bound to the unique positive root  of $p_{\delta}(\gamma)$.
\end{lemma}
\begin{proof}
See Appendix \ref{sec:proofs}.
\end{proof}


\subsubsection{The convergence factor \texorpdfstring{$\rho_a(\theta,\omega)$}{Lg} for a unit time interval}

\begin{lemma}\label{lem:sdd}
Recall the {\it unit-time interval} is the time in which the slowest processor (or the processor which takes the maximum time to perform an update) preforms an update. Then \begin{equation}
    \rho_a(\theta,\omega)\leq u(\theta,\omega,\delta_a)^{\delta_a}.
\end{equation}
\end{lemma}
\begin{proof}
The proof follows from Lemma \ref{lem:roc} and Lemma \ref{lem:kjd}.
\end{proof}

From \eqref{upper_bound} we have
\begin{equation} \label{ubrho}
    \rho_a(\theta,\omega) \leq \left(\frac{K_2+\frac{1}{\delta_a}}{1-K_1+\frac{1}{\delta_a}}\right).
\end{equation}
Let $\rho_{a_{\rm opt}}$ denote the optimal rate of convergence for the asynchronous SGD algorithm. Then
\begin{equation}\label{ub rho opt}
    \rho_{a_{opt}}\leq \min_{\theta,\omega}\left(\frac{K_2+\frac{1}{\delta_a}}{1-K_1+\frac{1}{\delta_a}}\right).
\end{equation}

\subsubsection{The iteration complexity}\label{sec:iter_complexity}
Let $\chi_a(\theta,\omega,\delta_a)$ denote the iteration complexity of the asynchronous SGD for some choice of $\theta$ and $\omega$, where 1 iteration denotes one unit time interval.~This ensures a fair comparison between synchronous and asynchronous SGD as synchronous parallel SGD performs single update of the parameter $x_t$ in a unit time interval. 
\begin{remark}
For simplicity, we are showing the iteration complexity results without the ``big Oh" ($\cO$) notation in this and all the subsequent sections. 
\end{remark}
 Then
\begin{align*}
    \chi_a(\theta,\omega,\delta_a) = \frac{1}{1-\rho_a(\theta,\omega)}.
\end{align*}
Notice that, for iteration complexity, the above expression holds only when $\rho_a(\theta,\omega)<1$.
Let $\chi_{a_{opt}}(\delta_a)$ denote the best possible iteration complexity of the asynchronous algorithm, that is,
\begin{equation}\label{eq:paos}
    \chi_{a_{opt}}(\delta_a)=\min_{\theta,\omega}\chi_a(\theta,\omega,\delta_a)
\end{equation}
which implies
\begin{equation}
    \chi_{a_{opt}}(\delta_a)=\frac{1}{1-\rho_{a_{opt}}}.
\end{equation}
From \eqref{ub rho opt} we have
\begin{equation}
    \chi_{a_{opt}}(\delta_a) \leq \min_{\theta,\omega}\left(\frac{1-K_1+\frac{1}{\delta_a}}{1-K_1-K_2}\right).
\end{equation}
Denote
\begin{equation}\label{defU}
    U(\theta,\omega)= \frac{1-K_1+\frac{1}{\delta_a}}{1-K_1-K_2}.
\end{equation}
For an arbitrary choice of $\theta$ and $\omega$, $U(\theta,\omega)$ denotes an upper bound on the iteration complexity of the asynchronous SGD with $\theta$ and $\omega$ set as the parameters of the algorithm. One should again note that the above expression for $U(\theta,\omega)$ is feasible only when $K_1+K_2<1$. In the next section, we show that we can find optimal ($\theta$, $\omega$) for \textbf{Case 1}, and a good combination of ($\theta$, $\omega$) for \textbf{Case 2}. 
Therefore,
\begin{equation}
    \chi_{a_{opt}}(\delta_a) \leq \min_{\theta,\omega}U(\theta,\omega).
\end{equation}

\begin{remark}\label{rem:chs}
In this paper, our goal is to compare between $\chi_{s_{opt}}(\tau)$ (as defined in Theorem \ref{minibatch}) and $\chi_{a_{opt}}(\delta_a)$.
\end{remark}

\subsubsection{The optimal \texorpdfstring{$\omega$}{Lg}}
\begin{lemma}\label{rem:oqx}
Assume that we have two choices of $\omega$, say, $\omega_1$ and $\omega_2$, such that for all $\theta\in[0,1]$, both $K_1$ and $K_2$ are smaller for $\omega_1$ than $\omega_2$. Then $\omega_1$ is a better choice than $\omega_2$.
\end{lemma}

\begin{proof}
See Appendix \ref{sec:proofs}. 
\end{proof}

\begin{lemma}\label{lem:sjk}
The optimal $\omega$ for all $\theta$ in $[0,1]$ in both the cases (See Theorem \ref{recurrence}) lies in the range $[1,\omega^\star]$.
\end{lemma}

\begin{proof}
See Appendix \ref{sec:proofs}. 
\end{proof}

We now try to find the minimum value of the upper bound $U(\theta,\omega)$ for $\theta\in[0,1]$ and $\omega\in[1,\omega^\star]$.

\begin{figure}
    \centering
    \begin{subfigure}{0.49\textwidth}
    \includegraphics[width=\textwidth]{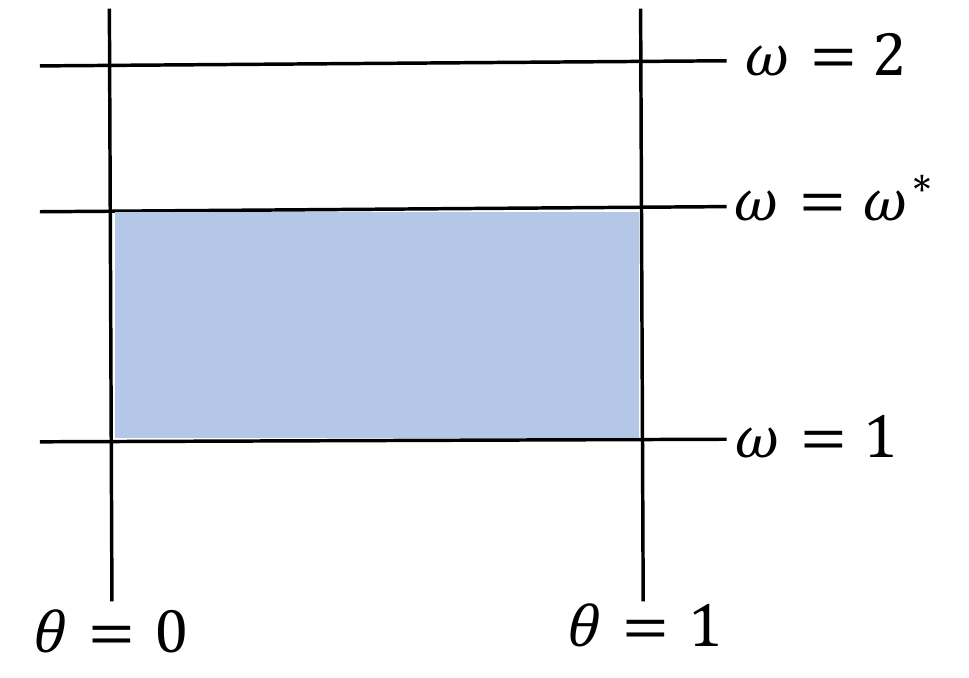}
        \caption{}
      \label{1a}
      \end{subfigure}
    \begin{subfigure}{0.49\textwidth}
    \includegraphics[width = \textwidth]{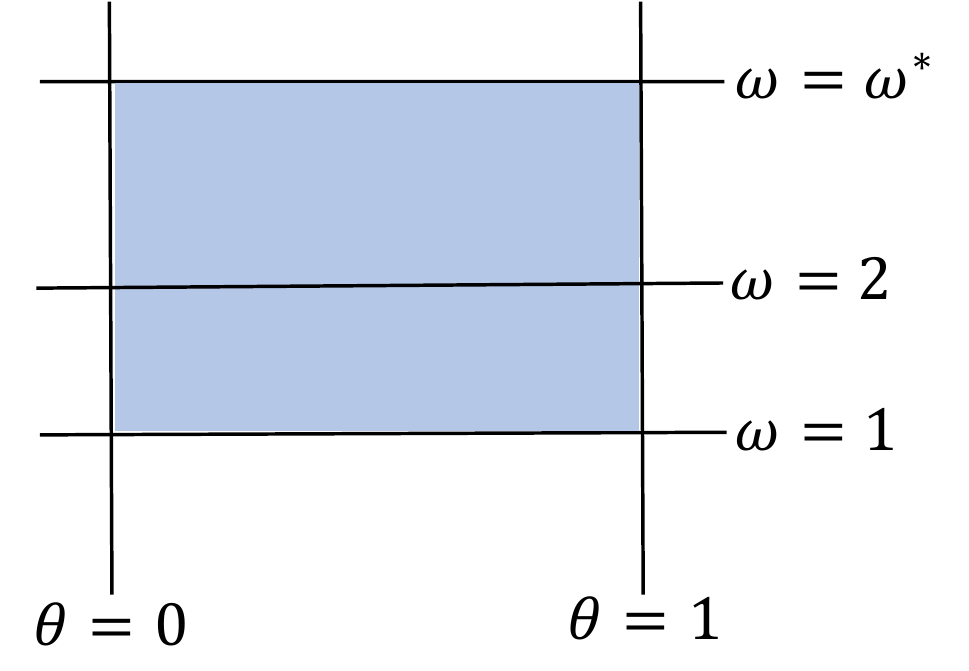}
    \caption{}   
    \label{1b}
  \end{subfigure} 
  \vspace{-15pt}
 \caption{\small{The region of interest $(\theta, \omega)=[0,1]\times [1,\omega^\star] $ is shaded-- (a) Case 1: $\omega^\star\leq 2$, that is, $\lambda^+_{\min}+\lambda_{\max} \geq 1$, (b) Case 2: $\omega^\star\geq 2$, that is, $\lambda^+_{\min}+\lambda_{\max} \leq 1$.}}
    \label{roi}
\end{figure}

\subsubsection{Upper bounding the condition number for Case 1}
As proved in Lemma \ref{lem:sjk}, we only need to focus in the region $\omega\in[1,\omega^\star]$. Please see Figure \ref{1a}. From \eqref{scenario_a_1} we have:

\begin{equation}\label{intitial_interval}
   K_1(\theta,\omega) \eqdef  (1-\theta)^2 + \theta(1-\theta)(1-\omega\lambda^+_{\min}), \qquad K_2(\theta,\omega) \eqdef \theta^2 \left(1-\omega(2-\omega)\lambda_{\min}^+\right) + \theta(1-\theta)(1-\omega\lambda^+_{\min}) .
\end{equation}
Substituting the values of $K_1$ and $K_2$ in \eqref{defU}, we have
\begin{equation}\label{ub_initial_interval}
      U(\theta,\omega)=\frac{\theta\left(1+\left(1-\theta\right)\omega\lambda^+_{\min}\right)+\frac{1}{\delta_a}}{\theta\omega\left(2-\theta\omega\right)\lambda^+_{\min}}.
\end{equation}
Now we compute $\nabla U$ and find:
\begin{equation}
     \frac{\partial U}{\partial\theta}=\frac{\theta^2\omega\delta_a\left(1+\left(\omega-2\right)\lambda^+_{\min}\right)+2\left(\theta\omega-1\right)}
    {\theta^2\omega\delta_a\left(2-\theta\omega\right)^2\lambda^+_{\min}}.
\end{equation}
Setting  $ \frac{\partial U}{\partial\theta}=0$ we have
\begin{align}
   \theta^2\omega\delta_a\left(1+\left(\omega-2\right)\lambda^+_{\min}\right)+2\left(\theta\omega-1\right)=0,
 \end{align}
which implies 
\begin{align}
  \theta^2\left(\omega\delta_a+\omega^2\delta_a\lambda^+_{\min}-2\omega\delta_a\lambda^+_{\min}\right)+ \theta\left(2\omega\right)-2=0. \label{partial_U_theta_zero}
\end{align}
Similarly,
\begin{equation}
       \frac{\partial U}{\partial\omega}=\frac{2\delta_a\theta\left(\theta\omega-1\right)+\delta_a\theta^2\omega^2\left(1-\theta\right)\lambda^+_{\min}+2\left(\theta\omega-1\right)}
    {\theta\omega^2\delta_a\left(2-\theta\omega\right)^2\lambda^+_{\min}},
\end{equation}
and setting $ \frac{\partial U}{\partial\omega}=0$ gives
\begin{align}
    2\delta_a\theta\left(\theta\omega-1\right)+\delta_a\theta^2\omega^2\left(1-\theta\right)\lambda^+_{\min}+2\left(\theta\omega-1\right)=0,
    \end{align}
 which implies 
 \begin{eqnarray}
   \theta^3\left(-\omega^2\delta_a\lambda^+_{\min}\right)+ \theta^2\left(\omega^2\delta_a\lambda^+_{\min}+2\omega\delta_a\right)+\theta\left(2\omega-2\delta_a\right)-2=0.
    \label{partial_U_omega_zero}
    \end{eqnarray}
For the gradient $\nabla U$ to vanish inside the region of interest ($\omega\in(1,\omega^\star)$ and $\theta\in(0,1)$), we want both equations (\ref{partial_U_theta_zero}) as well as (\ref{partial_U_omega_zero}) to hold. We try to find out the solutions of:  $(\ref{partial_U_theta_zero})-(\ref{partial_U_omega_zero})=0$ as any solution satisfying both equations (\ref{partial_U_theta_zero}) and (\ref{partial_U_omega_zero}) should also satisfy $(\ref{partial_U_theta_zero})-(\ref{partial_U_omega_zero})=0$.
Therefore, we have
\begin{align}
    (\ref{partial_U_theta_zero})-(\ref{partial_U_omega_zero}) &\ = \theta^3\left(\omega^2\delta_a\lambda^+_{\min}\right)+ \theta^2\left(-2\omega\delta_a\lambda^+_{\min}-\omega\delta_a\right)+2\theta\delta_a\\
    &\ = \delta_a\theta\underbrace{\left(\lambda^+_{\min}\theta^2\omega^2-\left(2\lambda^+_{\min}+1\right)\theta\omega+2\right)}_{\text{A quadratic in }\theta\omega}.
\end{align}
One solution of $(\ref{partial_U_theta_zero})-(\ref{partial_U_omega_zero})=0$ from above is $\theta=0$, which does not lie inside the region of interest. Next we focus on the quadratic:
\begin{align}
      \lambda^+_{\min}\theta^2\omega^2-\left(2\lambda^+_{\min}+1\right)\theta\omega+2 &\ =0,
    \end{align}
    that is,
   \begin{align}
     \lambda^+_{\min}\left(\theta\omega-\frac{1}{\lambda^+_{\min}}\right)\left(\theta\omega-2\right) &\ =0,
\end{align}
which gives the solution $\theta\omega=\frac{1}{\lambda^+_{\min}}$ and $\theta\omega=2$. The maximum value of $\theta\omega$ attainable in the region of interest is $\omega^\star$, for $\theta=1$ and $\omega=\omega^\star$. Hence $\theta\omega=\frac{1}{\lambda^+_{\min}}$ is not attainable inside the region of interest. Also, as we are dealing with Case 1, $\omega^\star\leq2$, hence $\theta\omega=2$ is also not attainable inside the region of interest. Thus,$\nabla U$ does not vanish inside the region of interest.  Therefore, the minima lies on the boundary of the region. We now discuss following four boundary cases:

\subsubsection{\texorpdfstring{$\theta=0$}{Lg}}
$U$ is not defined. 

\subsubsection{\texorpdfstring{$\theta=1$}{Lg}}
We have
\begin{align*}
      U(1,\omega)=\underbrace{\frac{1+\frac{1}{\delta_a}}{\omega(2-\omega)\lambda^+_{\min}}}_{\text{minimized at }\omega=1}\implies U(1,1)=\frac{1+\frac{1}{\delta_a}}{\lambda^+_{\min}}.
\end{align*}
We note that $U(1,\omega)>\frac{1}{\lambda^+_{\min}}$ for all $\delta_a$.
\subsubsection{\texorpdfstring{$\omega=1$}{Lg}}
Let $\theta_{1}= \underset{\theta}{\mathrm{argmin\hspace{0.1em}}}U(\theta,1)$. Then $\theta_1$ is the solution of equation (\ref{partial_U_theta_zero}) at $\omega=1$. Substituting $\omega=1$ in (\ref{partial_U_theta_zero}) we get the following quadratic in $\theta$
\begin{equation}
     \theta^2\delta_a(1-\lambda^+_{\min}) +2\theta-2=0. 
\end{equation}
As we want $\theta$ in $[0,1]$ and find
\begin{equation}
      \theta_1= \frac{\sqrt{1+2\delta_a(1-\lambda_{\min}^+)}-1}{\delta_a(1-\lambda_{\rm min}^+)}.
\end{equation}
Therefore,
\begin{equation}
      U(\theta_1,1)=\frac{\frac{3}{4}+\frac{1+\sqrt{1+2\delta_a(1-\lambda_{\min}^+)}}{4\delta_a}}{\lambda_{\rm min}^+}.
\end{equation}
Note that for $\delta_a\geq4$, $U(\theta_1,1)\leq \frac{1}{\lambda^+_{\min}}$. Thus, when $c\tau\geq4$, the asynchronous parallel method performs better than the basic method for Case 1.
\subsubsection{\texorpdfstring{$\omega=\omega^\star$}{Lg}}
Let $\theta_{\omega^\star}= \underset{\theta}{\mathrm{argmin\hspace{0.1em}}}U(\theta,\omega^\star)$. Then we get $\theta_{\omega^\star}$ as:
\begin{equation}
      \theta_{\omega^\star}=\frac{k(\sqrt{1+\delta_a(2-k)}-1)}{\delta_a(2-k)}\qquad (k=\lambda_{\min}^+ +\lambda_{\max}).
\end{equation}
Therefore,
\begin{eqnarray}
      U(\theta_{\omega^\star},\omega^\star) &=& \frac{\frac{k}{4}+\frac{\lambda_{\min}^+}{2}+\frac{\sqrt{1+\delta_a(2-k)}+2(\delta_a+1+\delta_a(1-k)\lambda_{\min}^+)}{2\delta_a\sqrt{1+\delta_a(2-k)}}}{\lambda_{\min}^+}\\
    &=& \frac{\frac{3\lambda_{\min}^++\lambda_{\max}}{4}+\frac{\sqrt{1+\delta_a(2-k)}+2(\delta_a+1+\delta_a(1-k)\lambda_{\min}^+)}{2\delta_a\sqrt{1+\delta_a(2-k)}}}{\lambda_{\min}^+}.\label{eq:ncj}
\end{eqnarray}
\begin{lemma}\label{lem:mnv}
For Case 1, we get the iteration complexity
\begin{equation}
   \chi_{a_{opt}}(\delta_a) \leq \min\left(U(\theta_1,1),U(\theta_{\omega^\star},\omega^\star)\right).
\end{equation}
\end{lemma}

\subsubsection{Upper bounding the condition number for Case 2}
We split the interval into two parts: $\omega\in[1,2]$ and $\omega\in[2,\omega^\star]$. Please see Figure \ref{1b}.

\subsubsection{Part 1: \texorpdfstring{$\omega\in[1,2]$}{Lg}}
$K_1(\theta,\omega)$ and $K_1(\theta,\omega)$ are the same as described in equation \eqref{intitial_interval}. Similar to the Case 1, we find that $\nabla U$ does not vanish inside the region $\theta\in(0,1)$ and $\omega\in(1,2)$.
Therefore, the minima lies on the boundary of the region.
We now discuss four boundary cases:

\paragraph{\texorpdfstring{$\theta=0$}{Lg}}
$U$ is not defined. 

\paragraph{\texorpdfstring{$\theta=1$}{Lg}}
We have
\begin{align*}
      U(1,\omega)=\underbrace{\frac{1+\frac{1}{\delta_a}}{\omega(2-\omega)\lambda^+_{\min}}}_{\text{minimized at }\omega=1}\implies U(1,1)=\frac{1+\frac{1}{\delta_a}}{\lambda^+_{\min}}.
\end{align*}
We note that $U(1,\omega)>\frac{1}{\lambda^+_{\min}}$ for all $\delta_a$.
\paragraph{\texorpdfstring{$\omega=1$}{Lg}}
Let $\theta_{1}= \underset{\theta}{\mathrm{argmin\hspace{0.1em}}}U(\theta,1)$. Then we get $\theta_{1}$ as:
\begin{equation}
      \theta_1= \frac{\sqrt{1+2\delta_a(1-\lambda_{\min}^+)}-1}{\delta_a(1-\lambda_{\min}^+)}.
\end{equation}
Therefore,
\begin{equation}
      U(\theta_1,1)=\frac{\frac{3}{4}+\frac{1+\sqrt{1+2\delta_a(1-\lambda_{\min}^+)}}{4\delta_a}}{\lambda_{\min}^+}.
\end{equation}
We note that for $\delta_a\geq4$, $U(\theta_1,1)\leq\frac{1}{\lambda^+_{\min}}$.

\paragraph{\texorpdfstring{$\omega=2$}{Lg}}
Let $\theta_{2}= \underset{\theta}{\mathrm{argmin\hspace{0.1em}}}U(\theta,2)$. Then we get $\theta_{2}$ as:
\begin{equation}
    \theta_2=\frac{\sqrt{\delta_a+1}-1}{\delta_a}.
\end{equation}
Therefore,
\begin{equation}\label{eq:opm}
      U(\theta_2,2)=\frac{\frac{1}{4}+\frac{\lambda_{\min}^+}{2}+\frac{1+\sqrt{\delta_a+1}}{2\delta_a}}{\lambda_{\min}^+}.
\end{equation}
We note that  for $\delta_a\geq3$, $U(\theta_2,2)\leq\frac{1}{\lambda^+_{\min}}$. Thus, for $c\tau\geq3$, the asynchronous parallel method performs better than the basic method for Case 2.
\begin{lemma}\label{lem:ysv}
For Case 2, 
\begin{equation}
     U(\theta_2,2)\leq U(\theta_1,1)\leq U(1,1) \quad \forall \delta_a\geq4.
\end{equation}
\end{lemma}
\begin{proof}
See Appendix \ref{sec:proofs}. 
\end{proof}
Thus, in Case 2, the optimal stepsize is $\omega=2$ for $\omega\in[1,2]$. This leads us to the following result:
\begin{lemma}\label{lem:ysv2}
For Case 2, the iteration complexity
\begin{equation}
    \chi_{a_{opt}}(\delta_a) \leq \frac{\frac{1}{4}+\frac{\lambda_{\min}^+}{2}+\frac{1+\sqrt{\delta_a+1}}{2\delta_a}}{\lambda_{\min}^+}.
\end{equation}
\end{lemma}

\subsubsection{Part 2: \texorpdfstring{$\omega\in[2,\omega^\star]$}{Lg}}

\begin{table}
\centering
\begin{tabular}{|c |c |c |c |c |c |}     
 \hline
$\lambda_{\min}^+$ &    $\lambda_{\max}$  & $ c$   &  $k=\lambda_{\min}^++\lambda_{\max}$   & $\kappa$ &   $\tau$ \\
   \hline
   $10^{-1}$   &    0.9  &    1  &  1  &       9  &     5 \\
    \hline
      $10^{-2}$  &    0.99  &    1   & 1   &     99  &     4 \\
      \hline
     $10^{-3}$  &   0.999  &    1  &  1    &   999 &      4 \\
     \hline
    $10^{-4}$  &  0.9999  &    1  &  1  &    9999 &      4 \\
    \hline
       $2\times 10^{-1}$   &    0.8   &   1  &  1    &     4  &     7 \\
        \specialrule{.2em}{.1em}{.1em} 
     $10^{-1}$ &   0.9  &  1.5  &  1    &    9   &     3 \\
     \hline
    $10^{-2}$  &  0.99   & 1.5  &  1    &   99    &    3 \\
     \hline
	 $10^{-3}$  &  0.999   & 1.5  &  1    &   999    &    3 \\
     \hline
      $10^{-4}$  &  0.9999   & 1.5  &  1    &   9999    &    3 \\
     \hline
    $2\times 10^{-1}$ &   0.8  &  1.5  &  1     &   4    &    5 \\
     \specialrule{.2em}{.1em}{.1em} 
    $10^{-1}$   &    0.9  &    2  &  1      &     9   &      3 \\
      \hline
      $10^{-2}$  &      0.99   &     2  &    1    &      99    &     2 \\
       \hline
     $10^{-3}$    &   0.999    &    2   &   1      &   999   &      2 \\
      \hline
    $10^{-4}$  &    0.9999   &     2   &   1    &    9999    &     2 \\
     \hline
    $2\times 10^{-1}$     &    0.8   &     2   &   1   &        4    &     4 \\
        \hline
 \end{tabular}
 \caption{$k=\lambda_{\min}^++\lambda_{\max}\ge 1$. Minimum number of processors $\tau$ required for asynchronous SGD to have better iteration complexity than synchronous SGD.}\label{table:case1}
 \end{table}

We were not able to find the optimal stepsize $\theta$ for  Case 2 in  $\omega\in[2,\omega^\star]$. However, the iteration complexity for $\omega=2$ is a significant result. 

\subsection{Comparing the iteration complexity with the synchronous parallel method}\label{sec:lak}
We know from \eqref{eq:ajd} the best iteration complexity $\chi_{s_{opt}}(\tau)$ of the synchronous parallel method is
$$
        {\chi_{s_{opt}}(\tau)= \frac{\frac{1}{\tau}+\left(1-\frac{1}{\tau}\right)\lambda_{\rm max}}{\lambda_{\rm min}^+}.}
$$
Rearranging the above equation we get
\begin{equation}\label{eq:kbd}
     \chi_{s_{opt}}(\tau)= \frac{\lambda_{\rm max}+\frac{1-\lambda_{\rm max}}{\tau}}{\lambda_{\rm min}^+}.
\end{equation}
Therefore,
\begin{equation}\label{eq:nbc}
      \lim_{\tau\to\infty}\chi_{s_{opt}}(\tau)= \frac{\lambda_{\rm max}}{\lambda_{\rm min}^+}.
\end{equation}

\subsubsection{Case 1}
Substituting $\delta_a=c\tau$ in equation \eqref{eq:ncj} we get
\begin{equation}\label{eq:wop}
    \chi_a(\theta_{\omega^\star},\omega^\star,\tau,c)=\frac{\frac{3\lambda_{\min}^++\lambda_{\rm max}}{4}+\frac{\sqrt{1+c\tau(2-k)}+2(c\tau+1+c\tau(1-k)\lambda_{\min}^+)}{2c\tau\sqrt{1+c\tau(2-k)}}}{\lambda_{\min}^+}. \qquad{(k=\lambda_{\min}^+ +\lambda_{\rm max})}.
\end{equation}
Therefore,
\begin{equation}\label{eq:mkl}
    \lim_{\tau\to\infty}\chi_a(\theta_{\omega^\star},\omega^\star,\tau,c)=\frac{\frac{3}{4}\lambda_{\min}^++\frac{\lambda_{\rm max}}{4}}{\lambda_{\min}^+}.
\end{equation}
Thus, comparing equations \eqref{eq:nbc} and \eqref{eq:mkl}, we find out that asymptotically, asynchronous SGD has a tighter iteration complexity than the synchronous parallel method  in Case 1~(when $(\lambda_{\min}^+ +\lambda_{\rm max})\in[1,2]$).

\subsubsection{Case 2}
Similarly for Case 2, from \eqref{eq:opm} we get 
\begin{equation}
    \chi_a(\theta_2,2,\tau,c)= \frac{\frac{1}{4}+\frac{\lambda_{\min}^+}{2}+\frac{1+\sqrt{c\tau+1}}{2c\tau}}{\lambda_{\min}^+}.
\end{equation}
Therefore,
\begin{equation}\label{eq:oip}
     \lim_{\tau\to\infty}\chi_a(\theta_2,2,\tau,c)=\frac{\frac{1}{4}+\frac{\lambda_{\min}^+}{2}}{\lambda_{\min}^+}.
\end{equation}
Finally, comparing equations \eqref{eq:kbd} and \eqref{eq:oip}, we find out that asymptotically, asynchronous SGD enjoys a tighter iteration complexity than synchronous parallel method Case 2~(when $(\lambda_{\min}^+ +\lambda_{\rm max})\in[0,1]$ and $\frac{1}{4}+\frac{ \lambda_{\rm min}^+}{2}\leq \lambda_{\rm max}$).
\begin{table}
\centering
\begin{tabular}{|c |c |c |c |c |c |}     
 \hline
$\lambda_{\min}^+$ &    $\lambda_{\max}$  & $ c$   &  $k=\lambda_{\min}^++\lambda_{\max}$   & $\kappa$ &   $\tau$ \\
   \hline

      $10^{-2}$   &  0.4  &    1   &    0.41    &   40   &     12\\
       \hline
      $10^{-2}$   &  0.3  &    1    &   0.31    &    30    &   116\\
       \hline
      $10^{-2}$  &  0.27  &    1    &   0.28   &     27    &  1082\\
       \hline
      $10^{-2}$  &  0.26   &   1    &   0.27    &    26   &   9905\\
       \hline
     $10^{-3}$  &   0.4   &   1   &   0.401    &   400  &      11\\
      \hline
     $10^{-3}$  &   0.3   &   1    &  0.301    &   300  &      95\\
      \hline
     $10^{-3}$  &  0.27  &    1    &  0.271    &   270   &    635\\
      \hline
     $10^{-3}$  &  0.26  &    1    &  0.261   &    260    &  2721\\
      \hline
    $10^{-4}$ &    0.4  &    1   &  0.4001  &    4000   &     11\\
     \hline
    $10^{-4}$  &   0.3   &   1   &  0.3001    &  3000     &   94\\
     \hline
    $10^{-4}$ &   0.27  &    1  &   0.2701   &   2700    &   606\\
     \hline
    $10^{-4}$  &  0.26   &   1  &   0.2601   &   2600    &  2478\\
     \hline
    $10^{-5}$ &   0.4   &   1  &  0.40001  &   40000    &    11\\
      \hline
    $10^{-5}$   &  0.3   &   1   & 0.30001   &  30000     &   93\\
      \hline
   $10^{-5}$  &  0.27  &    1  &  0.27001  &   27000    &   604\\
      \hline
  $10^{-5}$ &  0.26  &    1  &  0.26001 &   26000   &   2456\\
      \specialrule{.2em}{.1em}{.1em} 
 	$10^{-2}$   &  0.4   & 1.5    &   0.41     &   40     &    5     \\
 	\hline
     
      $10^{-2}$   &  0.3  &  1.5    &   0.31   &     30     &   66\\
      \hline
      
      $10^{-2}$   &  0.27   &  1.5   &     0.28    &     27    &    688\\
     \hline

      $10^{-2}$   &  0.26   &  1.5   &     0.27     &    26    &   6504\\
     \hline
	$10^{-3}$  &   0.4  &  1.5  &    0.401   &    400     &    5\\
	\hline
	
	$10^{-3}$  &   0.3  &  1.5    &  0.301    &   300   &     54 \\
	\hline	
	
     $10^{-3}$   &  0.27   &  1.5   &    0.271    &    270    &    398\\
     \hline

     $10^{-3}$   &  0.26   &  1.5    &   0.261      &  260     &  1761\\
     \hline
	
	$10^{-4}$  &   0.4 &   1.5   &  0.4001  &    4000      &   5	\\
	\hline
	
	$10^{-4}$  &   0.3  &  1.5  &   0.3001   &   3000     &   52	\\
	\hline
	
    $10^{-4}$    & 0.27    & 1.5    &  0.2701    &   2700    &    379\\
     \hline

    $10^{-4}$   &  0.26   &  1.5    &  0.2601    &   2600    &   1602\\
     \hline
     
     $10^{-5}$  &   0.4   & 1.5  &  0.40001   &  40000    &     5\\
     \hline
	
	$10^{-5}$  &   0.3   & 1.5   & 0.30001   &  30000  &      52\\
	\hline	
	
     $10^{-5}$    & 0.27    & 1.5   &  0.27001   &   27000    &    377\\
     \hline

     $10^{-5}$    & 0.26    & 1.5    & 0.26001   &   26000     &  1587\\
     \specialrule{.2em}{.1em}{.1em} 
      
      $10^{-2}$  &    0.4    &   2    &    0.41    &     40   &       2\\
      \hline
      $10^{-2}$  &    0.3   &    2   &     0.31    &     30    &     40\\
      \hline
      $10^{-2}$ &    0.27  &     2  &      0.28    &     27  &      491\\
      \hline
      $10^{-2}$   &  0.26 &      2    &    0.27    &     26    &   4803\\
      \hline
     $10^{-3}$  &    0.4   &    2   &    0.401   &     400   &       2\\
     \hline
     $10^{-3}$  &    0.3  &     2   &    0.301   &     300   &      31\\
     \hline
     $10^{-3}$  &   0.27 &      2   &    0.271  &      270   &     278\\
     \hline
     $10^{-3}$  &   0.26  &     2   &    0.261  &      260  &     1281\\
     \hline
    $10^{-4}$  &    0.4  &     2  &    0.4001 &      4000    &      2\\
    \hline
    $10^{-4}$ &     0.3 &      2   &   0.3001  &     3000   &      31\\
    \hline
    $10^{-4}$  &   0.27   &    2  &    0.2701   &    2700   &     265\\
    \hline
    $10^{-4}$ &    0.26  &     2  &    0.2601   &    2600    &   1163\\
    \hline
     $10^{-5}$  &    0.4  &     2  &   0.40001  &     40000   &       2\\
     \hline
   $10^{-5}$  &    0.3  &     2   &  0.30001 &     30000    &     31\\
     \hline
     $10^{-5}$ &   0.27   &    2 &    0.27001  &    27000     &   264\\
     \hline
    $10^{-5}$ &   0.26  &     2  &   0.26001   &   26000    &   1152\\
     \hline
    \end{tabular}
 \caption{$k=\lambda_{\min}^++\lambda_{\max}\le 1$. Minimum number of processors $\tau$ required for asynchronous SGD to have better iteration complexity than synchronous SGD.}\label{table:case2}
     \end{table}

\subsection{Non asymptotic comparison}
As the above comparisons are asymptotic, we want to find out when exactly the asynchronous algorithm is better to use compare to its synchronous counterpart. For given combinations of $\lambda_{\min}^+$, $\lambda_{\max}$ and $c$, we computed the minimum number of processors, $\tau$, for which the asynchronous SGD algorithm has better iteration complexity (equation \eqref{eq:paos}) than its synchronous counterpart (equation \eqref{eq:ajd}). We have considered $c\in\{1,1.5,2\}$, and various values of interest for $\lambda_{\min}^+$ and $\lambda_{\max}$. Denote the condition number, $\kappa = \frac{\lambda_{\rm max}}{\lambda_{\min}^+}$.

\subsubsection{Case 1}
Table \ref{table:case1} refers to the first case, when $k=\lambda_{\min}^++\lambda_{\max}\ge 1$. We notice that the asynchronous SGD performs better than synchronous for small minimum number of processors $\tau$ even for highly ill conditioned problems in Case 1.  

\subsubsection{Case 2}   
Table \ref{table:case2} refers to the second case, when  $k=\lambda_{\min}^++\lambda_{\max}\le 1$.~We note that the synchronous method is always better when $\frac{1}{4}+\frac{ \lambda_{\rm min}^+}{2} > \lambda_{\rm max}$. ~Thus, in order to consider other cases, all the linear systems below satisfy $\frac{1}{4}+\frac{ \lambda_{\rm min}^+}{2} \leq \lambda_{\rm max}$. We vary $\lambda^+_{\min}$ in log-scale, $\lambda_{\max}\in\{0.4,0.3,0.27,0.26\}$, and $c\in\{1,1.5,2\}$.

\begin{itemize}

\item \textbf{Effect of varying $\lambda_{\max}$.} We observe that asynchronous SGD can perform better than the synchronous SGD in a reasonable number of processors, $\tau$, if $\lambda_{\max}$ is reasonably larger than $\frac{1}{4}+\frac{ \lambda_{\rm min}^+}{2}.$ For example, when $\lambda_{\max}=0.4$, the minimum number of processors, $\tau=2$ for $c=2$ and $\tau=11$ for $c=1$ even for highly ill-conditioned problems (for both cases, $\kappa=40000$). Fixing $c=2$, for each value of $\lambda^+_{\min}$, we observe that the minimum number of processors, $\tau$ increases from 2 for $\lambda_{\max}=0.4$ to values in thousands for $\lambda_{\max}=0.26$. 

\item  \textbf{Effect of varying $\lambda_{\min}^+$.} We observe that for the same $\lambda_{\max}$, the smaller the $\lambda_{\min}^+$ is, the smaller the minimum number of processors, $\tau$, is required. In other words, for the same $\lambda_{\max}$, asynchronous SGD requires less minimum number of processors, $\tau$ to beat its synchronous counterpart for more ill-conditioned problems. For example, for $\lambda_{\rm max}=0.27$ and $c=2,$ the minimum number of processors, $\tau=491$ when $\lambda_{{\rm min}^+}=10^{-2}$ (in this case, $\kappa=27$) and $\tau=264$ when $\lambda_{{\rm min}^+}=10^{-5}$ (in this case $\kappa=27000$). 

\item \textbf{Effect of varying $c$.} We also observe that increasing $c$ from $1$ to $2$ significantly brings down the minimum number of processors, $\tau$. This intuitively makes sense and in accordance to our set-up---as $c$ increases, the asynchronous parallel SGD performs more updates in a unit time interval. 
  
  \end{itemize}

\vspace{3pt}
\appendix 
\section{Proofs of the Theorems and Lemmas}\label{sec:proofs}
 \vspace{3pt}

 \paragraph*{\textbf{Proof of Theorem \ref{weak_recursion}}}
 Note that $\mS_{t-\delta}$ (and hence $\mZ_{t-\delta}$) is independent of both $x_{t-\delta}$ and $x_t$. Thus, taking expectation in \eqref{eq:jd8d8bs9jhisoiP} with respect to $\mS_{t-\delta}$ we get
$$\Exp_{\mS_{t-\delta} \sim \cD}\left[x_{t+1}-x_\star\right] = (1-\theta)(x_t-x_\star) + \theta (\mI -\omega \mB^{-1}\Exp_{\mS_{t-\delta}\sim \cD}\left[\mZ_{t-\delta}\right])(x_{t-\delta}-x_\star).
$$
Taking expectation again, and by using the tower property, we get
\begin{equation}\label{eq:s9gh8hf8ijhsiud}
\E{x_{t+1}-x_\star} = (1-\theta) \E{x_t-x_\star} + \theta (\mI-\omega \mB^{-1}\E{\mZ})\E{x_{t-\delta} - x_\star}.
\end{equation}
Pre-multiplying both sides of \eqref{eq:s9gh8hf8ijhsiud} by $\mB^{1/2}$ we find
\begin{align*}
\mB^{1/2}\mathbb{E}[x_{t+1}-x_{\star}]&=(1-\theta)\mB^{1/2}\mathbb{E}[x_{t}-x_{\star}]+\theta(\mI-\omega \mB^{-1/2}\mathbb{E}[\mZ] \mB^{-1/2})\mB^{1/2}\mathbb{E}[x_{t-\delta}-x_{\star}].
\end{align*}
Recalling that $\mB^{-1/2}\mathbb{E}[\mZ]\mB^{-1/2}=\mW$ has an eigenvalue decomposition $\mW=\mU\Lambda\mU^\top$ we obtain the desired result. 
\qedwhite

 \paragraph*{\textbf{Proof of Lemma \ref{lem:bs983bv78dv}}}
 The first part follows from \eqref{eq:jd8d8bs9jhisoiP}. The second part follows from the observation that ${\rm Im}(\mB^{-1}\mA^\top)$ is the $\mB$-orthogonal complement of the nullspace of $\mA$.
\qedwhite
 
 \paragraph*{\textbf{Proof of Theorem \ref{recurrence}}}We hereby provide the proof for $\omega\in[0,2]$. The proof for $\omega\geq2$ follows similarly, by using the upper bound of the inequality of Lemma \ref{lem:osohhd9u93}. By taking norms on both sides of \eqref{eq:jd8d8bs9jhisoiP} and then applying Lemma~\ref{eq:09s9hsoiuis907}, we get
\begin{eqnarray*}
\|x_{t+1}-x_{\star}\|^2_\mB &\overset{\eqref{eq:jd8d8bs9jhisoiP}}{=}& (1-\theta)^2 \|x_{t}-x_{\star}\|^2_\mB+\theta^2 \|(\mI-\omega \mB^{-1} \mZ_{t-\delta})(x_{t-\delta}-x_{\star})\|_\mB^2\\ && \qquad + 2(1-\theta)\theta \langle x_{t}-x_{\star},(\mI-\omega \mB^{-1}\mZ_{t-\delta})(x_{t-\delta}-x_{\star}) \rangle_\mB \\
&\overset{\text{Lemma}~\ref{eq:09s9hsoiuis907}}{=}&(1-\theta)^2 \|x_{t}-x_{\star}\|^2_\mB+\theta^2 \|x_{t-\delta}-x_{\star}\|_\mB^2 - 2\omega(2-\omega)\theta^2 f_{\mS_{t-\delta}}(x_{t-\delta})\\&& \qquad +2(1-\theta)\theta \langle x_{t}-x_{\star},(\mI-\omega \mB^{-1}\mZ_{t-\delta})(x_{t-\delta}-x_{\star})\rangle_\mB.    
\end{eqnarray*}
The above identity can be written as
\begin{eqnarray*}
\| r_{t+1}\|^2=(1-\theta)^2 \|r_t\|^2 + \theta^2 \| r_{t-\delta}\|^2 - 2\omega(2-\omega)\theta^2f_{\mS_{t-\delta}}(x_{t-\delta})\\
+2(1-\theta)\theta \langle r_t,(\mI-\omega \mB^{-1/2}\mZ_{t-\delta}\mB^{-1/2})r_{t-\delta}\rangle.    
\end{eqnarray*}
Conditioning on $x_t,\dots,x_{0}$, the only free random variable is $\mS_{t-\delta}$. Therefore, in view of Lemma~\ref{lem:osohhd9u93} and using the eigenvalue decomposition $\mB^{-1/2}\Exp[\mZ] \mB^{-1/2}= \mU\Lambda \mU^\top$, we  get the following bound on $C \eqdef \Exp\left[\| r_{t+1} \|^2 \;|\; x_t,\dots, x_0\right]$:
\begin{eqnarray*} C & = &  (1-\theta)^2 \| r_t \|^2 +\theta^2 \| r_{t-\delta} \|^2 -2\omega(2-\omega)\theta^2 f(x_{t-\delta})   +2(1-\theta)\theta r_t^\top (\mI - \omega \mB^{-1/2} \Exp[\mZ] \mB^{-1/2}) r_{t-\delta}\\
& \overset{\text{Lemma}~\ref{lem:osohhd9u93}}{\leq} & (1-\theta)^2 \| r_t \|^2 + \theta^2(1-\omega(2-\omega)\lambda_{\min}^+) \| r_{t-\delta} \|^2  + 2(1-\theta)\theta \underbrace{r_t^\top (\mI-\omega \mU\Lambda \mU^\top)r_{t-\delta}}_{D}.
\end{eqnarray*}
 Further, we have 
\begin{eqnarray*}
D &=&  r_t^\top (\mI-\omega \mU\Lambda \mU^\top)r_{t-\delta}\\
& = &  (\mU^\top r_t)^\top (\mI-\omega \Lambda )\mU^\top r_{t-\delta}
\\
&=& \sum_i (1-\omega \lambda_i )u_i^\top r_tu_i^{\top}r_{t-\delta}\\
&\overset{\text{Lemmas}~\ref{lem:hbs6763vs6} \text{ and } \ref{lem:bs983bv78dv}}{=}& \sum_{i:\lambda_i \neq 0} (1-\omega \lambda_i )u_i^{\top}r_t u_i^{\top}r_{t-\delta}\\
&\leq &  \sum_{i:\lambda_i \neq 0} |1-\omega \lambda_i | | u_i^{\top}r_t u_i^{\top}r_{t-\delta}| \\
&\leq &  \alpha(\omega) \sum_{i:\lambda_i \neq 0}  | u_i^{\top}r_t u_i^{\top}r_{t-\delta}| \\
& \overset{\text{(Cauchy-Schwarz)}}{\leq} & \alpha(\omega) \|r_t \| \|r_{t-\delta}\| \\
&\overset{\text{(AM-GM)}}\leq& \frac{\alpha(\omega) }{2} \left(\|r_t\|^2+\|r_{t-\delta}\|^2\right).
\end{eqnarray*}
Combining the bounds on $C$ and $D$, we get
\begin{eqnarray}\label{recursion_strong1}
C & \leq & (1-\theta)^2 \| r_t \|^2 + \theta^2(1-\omega(2-\omega)\lambda_{\min}^+) \| r_{t-\delta} \|^2  + (1-\theta)\theta \alpha(\omega)  \left(\|r_t\|^2+\|r_{t-\delta}\|^2\right) \notag \\
&=& \left[(1-\theta)^2+(1-\theta)\theta\alpha(\omega)\right] \|r_t\|^2\nonumber\\
&&+ \left[\theta^2 (1-\omega(2-\omega)\lambda_{\min}^+) + (1-\theta)\theta\alpha(\omega) \right] \|r_{t-\delta}\|^2. \label{inalpha}
\end{eqnarray}
The final result is obtained after we take full expectation and apply tower property.
\qedwhite

\paragraph*{\textbf{Proof of Lemma \ref{lem:jsh}}}
We have
\begin{equation}
    p_{\delta_a}(\gamma)- p_{\delta}(\gamma)=\gamma^\delta\left(\gamma-K_1\right)(\gamma^{\delta_a-\delta}-1)\leq0 \quad \forall\text{ } \gamma\in[K_1,1] \text{ and }  \delta\in[0,\delta_a].
\end{equation}
Thus, $p_{\delta_a}(\gamma)$ bounds $p_{\delta}(\gamma)$ from below in $\gamma\in[K_1,1]$, and hence its positive root $\varrho_{A}(\theta,\omega,\delta_a)$ is greater than or equal $\varrho_{A}(\theta,\omega,\delta)$.
\qedwhite

\paragraph*{\textbf{Proof of Lemma \ref{lem:roc}}}
 The inequality $ \rho_a(\theta,\omega)\leq\prod_{i=1}^{\delta_a}\varrho(\theta,\omega,\delta_i)$ follows from Theorem \ref{conv_strong}, and $\prod_{i=1}^{\delta_a}\varrho(\theta,\omega,\delta_i)\leq\varrho(\theta,\omega,\delta_a)^{\delta_a}$ follows from Lemma \ref{lem:jsh}. Together, we have \eqref{rate_of_convergence}. 
\qedwhite

\paragraph*{\textbf{Proof of Lemma \ref{lem:kjd}}}
We have$$p_{\delta}(0)\geq g_{\delta}(0)\;\;{\rm and}\;\;p_{\delta}(1)=g_{\delta}(1).$$
Let $q_{\delta}(\gamma)=p_{\delta}(\gamma)-g_{\delta}(\gamma)$ then 
$q_{\delta}(\gamma)=\gamma^{\delta+1}  - \left(1+\frac{1}{\delta}\right)\gamma^\delta +\frac{1}{\delta},
$
and $q_{\delta}'(\gamma)=\left(\delta+1\right)(\gamma-1)\gamma^{\delta-1}\leq 0$ for all $\gamma \in [0,1].$
Therefore, for all $\gamma \in [0,1],$ we have 
$
q_{\delta}(\gamma)\geq q_{\delta}(1),
$
which implies $
q_{\delta}(\gamma)\geq 0$, and finally,
$
 p_{\delta}(\gamma)-g_{\delta}(\gamma) \geq 0$ for all $\gamma \in [0,1]. 
$
Therefore, the root $u(\theta,\omega,\delta)= \left(\frac{K_2+\frac{1}{\delta}}{1-K_1+\frac{1}{\delta}}\right)^{\scalebox{1.2}{$\frac {1}{\delta}$}}$ of $g_{\delta}(\gamma)$
 is an upper bound to the unique positive root  of the characteristic polynomial $p_{\delta}(\gamma)$.
\qedwhite

%

\paragraph*{\textbf{Proof of Lemma \ref{rem:oqx}}}
We have for all $\theta$
\begin{align*}
    K_1(\theta,\omega_1) &\ \leq K_1(\theta,\omega_2)\\ \text{and}\quad
     K_2(\theta,\omega_1) &\ \leq K_2(\theta,\omega_2).
\end{align*}
Then by equation (\ref{rec3}), we have $\rho_{a}(\theta,\omega_1)\leq\rho_{a}(\theta,\omega_2) $ for all $\theta\in[0,1].$
Similarly, we see that by the definition of $U(\theta,\omega)$ in \eqref{defU} (provided $K_1+K_2<1$ in both choices of parameters)
$U(\theta,\omega_1)\leq U(\theta,\omega_2)$ for all $\theta\in[0,1].$
Hence the result. 
\qedwhite

\noindent\paragraph*{\textbf{Proof of Lemma \ref{lem:sjk}}}
Now we show how $K_1$ and $K_2$ behave for both the cases in the range $\omega\in[0,1]$ and $\omega\in[\omega^\star,\infty)$. Define $s(\omega)\eqdef(1-\omega(2-\omega)\lambda_{\min}^+)$ and $t(\omega)\eqdef(1-\omega(2-\omega)\lambda_{\max})$.

\textbf{Case 1:}
\begin{enumerate}[label=(\roman*)]
\item For $0\leq\omega\leq1$ and $\alpha(\omega)=1-\omega\lambda^+_{\min}$:\newline

\begin{tikzpicture}
    \draw (0,0) -- (4,0);
    \fill[black] (0,0) circle (0.6 mm) node[below] {$0$};
    \fill[black] (4,0) circle (0.6 mm) node[below] {$2$};
     \fill[black] (2,0) circle (0.6 mm) node[below] {$1$};
    \fill[black] (3,0) circle (0.6 mm) node[below] {$\omega^\star$};
    \fill[black] (1,0) circle (0.6 mm) node[below] {$\omega$};
\end{tikzpicture}

\begin{equation}
K_1(\theta,\omega) \eqdef  (1-\theta)(1-\theta + \theta\alpha(\omega)), \; K_2(\theta,\omega) \eqdef \theta (\theta \underbrace{(1-\omega(2-\omega)\lambda_{\min}^+)}_{s(\omega)} + (1-\theta)\alpha(\omega) ).
\end{equation} 
We see that both $\alpha(\omega)$ and $s(\omega)$ are monotonically decreasing in the interval $\omega\in[0,1]$. Hence, both $K_1$ and $K_2$ are monotonically decreasing in $\omega\in[0,1]$.
\item  For $\omega^\star\leq\omega\leq 2$ and $\alpha(\omega)=\omega\lambda_{\max}-1$:\newline

\begin{tikzpicture}
    \draw (0,0) -- (4,0);
    \fill[black] (0,0) circle (0.6 mm) node[below] {$0$};
    \fill[black] (4,0) circle (0.6 mm) node[below] {$2$};
     \fill[black] (2,0) circle (0.6 mm) node[below] {$1$};
    \fill[black] (2.5,0) circle (0.6 mm) node[below] {$\omega^\star$};
    \fill[black] (3,0) circle (0.6 mm) node[below] {$\omega$};
\end{tikzpicture}

\[K_1(\theta,\omega) \eqdef  (1-\theta)(1-\theta + \theta\alpha(\omega)), \; K_2(\theta,\omega) \eqdef \theta (\theta \underbrace{(1-\omega(2-\omega)\lambda_{\min}^+)}_{s(\omega)} + (1-\theta)\alpha(\omega) ).\]

Again, both $\alpha(\omega)$ and $s(\omega)$ are monotonically increasing in the interval $\omega\in[\omega^\star,2]$. Hence, both $K_1$ and $K_2$ are monotonically increasing in $\omega\in[\omega^\star,2]$.

\item For $\omega\geq 2$ and $\alpha(\omega)=\omega\lambda_{\max}-1$:\newline

\begin{tikzpicture}
   \draw (0,0) -- (4,0);
    \fill[black] (0,0) circle (0.6 mm) node[below] {$0$};
    \fill[black] (3,0) circle (0.6 mm) node[below] {$2$};
     \fill[black] (1.5,0) circle (0.6 mm) node[below] {$1$};
    \fill[black] (2.5,0) circle (0.6 mm) node[below] {$\omega^\star$};
    \fill[black] (4,0) circle (0.6 mm) node[below] {$\omega$};
\end{tikzpicture}

\[K_1(\theta,\omega) \eqdef  (1-\theta)(1-\theta + \theta\alpha(\omega)), \; K_2(\theta,\omega) \eqdef \theta (\theta \underbrace{(1-\omega(2-\omega)\lambda_{\max})}_{t(\omega)} + (1-\theta)\alpha(\omega) ).\]

Again, both $\alpha(\omega)$ and $t(\omega)$ are monotonically increasing in the interval $\omega\in[2,\infty)$. Hence, both $K_1$ and $K_2$ are monotonically increasing in $\omega\in[2,\infty)$.
\end{enumerate}

\textbf{Case 2:}

\begin{enumerate}[label=(\roman*)]
\item \hspace{1em} For $0\leq\omega\leq1$ and $\alpha(\omega)=1-\omega\lambda^+_{\min}$:\newline

\begin{tikzpicture}
    \draw (0,0) -- (4,0);
    \fill[black] (0,0) circle (0.6 mm) node[below] {$0$};
    \fill[black] (3,0) circle (0.6 mm) node[below] {$2$};
    \fill[black] (4,0) circle (0.6 mm) node[below] {$\omega^\star$};
    \fill[black] (2,0) circle (0.6 mm) node[below] {$1$};
    \fill[black] (1,0) circle (0.6 mm) node[below] {$\omega$};
\end{tikzpicture}

\begin{equation}
K_1(\theta,\omega) \eqdef  (1-\theta)(1-\theta + \theta\alpha(\omega)), \; K_2(\theta,\omega) \eqdef \theta (\theta \underbrace{(1-\omega(2-\omega)\lambda_{\min}^+)}_{s(\omega)} + (1-\theta)\alpha(\omega) ).
\end{equation}

Similar to case 1, both $\alpha(\omega)$ and $s(\omega)$ are monotonically decreasing in the interval $\omega\in[0,1]$. Hence, both $K_1$ and $K_2$ are monotonically decreasing in $\omega\in[0,1]$.

\item For $\omega\geq\omega^\star$ and $\alpha(\omega)=\omega\lambda_{\max}-1$:\newline

\begin{tikzpicture}
     \draw (0,0) -- (4,0);
    \fill[black] (0,0) circle (0.6 mm) node[below] {$0$};
    \fill[black] (2,0) circle (0.6 mm) node[below] {$2$};
    \fill[black] (3,0) circle (0.6 mm) node[below] {$\omega^\star$};
    \fill[black] (4,0) circle (0.6 mm) node[below] {$\omega$};
\end{tikzpicture}

$$K_1(\theta,\omega) \eqdef  (1-\theta)(1-\theta + \theta\alpha(\omega)), \; K_2(\theta,\omega) \eqdef \theta (\theta \underbrace{(1-\omega(2-\omega)\lambda_{\max})}_{t(\omega)} + (1-\theta)\alpha(\omega) ).$$
Again, both $\alpha(\omega)$ and $t(\omega)$ are monotonically increasing in the interval $\omega\in[\omega^\star,\infty)$. Hence, both $K_1$ and $K_2$ are monotonically increasing in $\omega\in[\omega^\star,\infty)$.
\end{enumerate}
Now, for both the cases, for all $\theta\in [0,1]$ we have the following:
\begin{itemize}
  \item [$\omega\leq1$ :] Both $K_1$ and $K_2$ monotonically decrease till $\omega=1$, then from Lemma \ref{rem:oqx}, $1$ is the optimal $\omega$ in $[0,1]$.~(Note that $\omega^\star$ is always greater than or equal to $1$) 
  \item [$\omega\geq\omega^\star$ :] Both $K_1$ and $K_2$ monotonically increase after $\omega=\omega^\star$, then from remark \ref{rem:oqx}, $\omega^\star$ is the optimal $\omega$ in $[\omega^\star,\infty)$.
\end{itemize}
Thus, the optimal $\omega$ for all $\theta\in[0,1]$  in both the cases lies in the range $[1,\omega_\star]$.
\qedwhite

\paragraph*{\textbf{Proof of Lemma \ref{lem:ysv}}}
The first inequality follows from the fact that $\lambda_{\min}^+ \leq \frac{1}{2}$ for Case 2 (as $\lambda_{\min}+\lambda_{\max}\leq1$). The second follows from the fact that $U(1,\omega)$ is always greater than $\frac{1}{\lambda_{\min}^+}$ for all $\delta_a$, whereas $U(\theta_1,1)$ is smaller than $\frac{1}{\lambda_{\min}^+}$ for $\delta_a\geq4$. 
\qedwhite

\bibliographystyle{plain}
\bibliography{references_asyn}

\section{Notation Glossary}\label{sec:notation}
\begin{tabular}{ |p{3cm}|p{11cm}|p{2cm} | }
 \hline
 \multicolumn{3}{|c|}{\textbf{The Basics}} \\
\hline
 $\mA,b$   & $m \times n$ matrix and $m \times 1$ vector defining the system $\mA x=b$    & Related to \\
 $\cL$ &   $\{x\; : \mA x=b\}$ (solution set of the linear system)  &  stochastic reformulation  \\
 $\mB$ & $n \times n$ symmetric positive definite matrix & of the linear  \\
 $\langle x , y \rangle_{\mB}$    & $x^\top \mB y$ ($\mB$-inner product)& system~(see \eqref{linear_system})\\
 $\|x\|_{\mB}$ &   $\sqrt{\langle x , x \rangle_{\mB}}$~(Norm induced by $\mB$-inner product) & \\
 $\mM^{\dagger}$ & Moore-Penrose pseudoinverse of matrix $\mM$  & 
 \\
 $\mS$ & a random real matrix with $m$ rows  & \\
 $\cD$ & distribution from which matrix $\mS$ is drawn & \\
 $\mH$ & $\mH=\mS(\mS{^\top}\mA\mB^{-1}\mA^{\top}\mS)^{\dagger}\mS^{\top}$ & \\
 $\mZ$ & $\mA^\top\mH\mA$ & \\
 Im$(\mM)$ & Image (range) space of matrix $\mM$ & \\
 Null$(\mM)$ & Null space of matrix $\mM$ & \\
 $\E{\cdot}$ & expectation& \\
 
 \hline
 \multicolumn{3}{|c|}{\textbf{Projections}}\\
 \hline
  $\Pi^{\mB}_{\cL}(x)$   & projection of $x$ onto the set $\cL$ in the $\mB$-norm   & \\
  $\mB^{-1}\mZ$ & projection matrix in the $\mB$-norm  onto ${\rm Im}(\mB^{-1}\mA^\top\mS)$ & \\
     
 \hline
 \multicolumn{3}{|c|}{\textbf{Optimization}}\\
 \hline
 $x_\star$ & a solution of the linear system $\mA x=b$ & Assumption \ref{assump1}\\
 $f_{\mS}, \nabla f_{\mS}, \nabla^2 f_{\mS}$  & stochastic function, its gradient and Hessian, respectively & \\
 $\cL_{\mS}$ & $\{x : \mS^\top \mA x = \mS^\top b\}$ (set of minimizers of $f_{\mS}$) & \\
 $f$ & $\E{f_{\mS}}$ & \\
 $\nabla f$ & gradient of f with respect to the \mB-inner product & \\
 $\nabla^2 f$ & $\mB^{-1}\E{\mZ}$ (Hessian of $f$ in the $\mB$-inner product) & \\
 \hline
 \multicolumn{3}{|c|}{\textbf{Eigenvalues}}\\
 \hline
 $\mW$ & $\mB^{-\frac{1}{2}}\E{\mZ}\mB^{-\frac{1}{2}}$ (psd matrix with the same spectrum as $\nabla^2f$) & \\
 $\lambda_1,\hdots,\lambda_n$ & eigenvalues of $\mW$& \\
 $\Lambda$ & \textbf{Diag}$(\lambda_1,\hdots,\lambda_n)$ (diagonal matrix of eigenvalues) & \\
 $\mU$ & $[u_1,\hdots, u_n]$ (eigenvectors of $\mW$) & \\
 $\mU\Lambda\mU^\top$ & eigenvalue decomposition of $\mW$ & \\
 $\lambda_{\max}, \lambda^+_{\min}$ & largest and smallest nonzero eigenvalues of $\mW$ & \\
 
 \hline
 \multicolumn{3}{|c|}{\textbf{Algorithms}}\\
 \hline
 $\theta$ & damping parameter & \\
 $\omega$ & stepsize / relaxation parameter & \\
 $\omega^\star$ & $\ddfrac{2}{(\lambda^+_{\min}+\lambda_{\max})}$ & Theorem   \ref{recurrence}\\
 $\tau$   & number of processors in the assembly   & \\
 $\delta$ & delay between the master and a particular worker processor & \\
 $\delta_a (=c\tau) $ & number of updates by the asynchronous parallel assembly in a unit time interval a.k.a the delay for the slowest worker processor (assumed to be constant) & \\
 
 & & \\
 & \textit{\rm For particular values of }$\theta$ , $\omega$ \textit{and} $\delta_a$: & \\
 
 $\alpha(\omega)$ & $\max_{i:\lambda_i>0}|1-\omega\lambda_i|$ & Theorem \ref{recurrence}\\
$K_1(\theta,\omega)$, $K_2(\theta,\omega)$ & coefficient of $\E{\|r_t\|^2]}$ and $\E{\|r_{t-\delta}\|^2}$ respectively in Theorem \ref{recurrence} & Theorem \ref{recurrence}\\
$\rho_a(\theta,\omega)$ & rate of convergence of the asynchronous parallel SGD in a unit time interval & \\
$\chi_a(\theta,\omega,\delta_a)$ & iteration complexity of the asynchrous parallel SGD & \\
\hline
\end{tabular}

\end{document}